\documentclass[12pt]{amsart}
\usepackage{amssymb,amsmath,amscd,latexsym,epsfig,color,hyperref,graphics}
\usepackage{bbm}
\usepackage[all,cmtip]{xy}

\usepackage[margin=1in]{geometry}

\newtheorem{theorem}{Theorem}[section]
\newtheorem{proposition}[theorem]{Proposition}
\newtheorem{corollary}[theorem]{Corollary}

\newtheorem{lemma}[theorem]{Lemma}

\theoremstyle{definition}
\newtheorem{definition}[theorem]{Definition}
\newtheorem{remark}[theorem]{Remark}
\newtheorem{example}[theorem]{Example}

\newcommand{\kin}{\mathcal{O}_{in}}
\newcommand{\kout}{\mathcal{O}_{out}}

\def\RR{\ensuremath{\mathbb{R}}}

\def\rank{\ensuremath{\textup{rank}}}
\def\max{\ensuremath{\textup{max}}}
\def\min{\ensuremath{\textup{min}}}

\def\SOC{\ensuremath{\textup{SOC}}}
\def\PSD{\ensuremath{\mathcal{S}_{+}}}

\def\Inn{\ensuremath{\textup{Inn}}}
\def\Out{\ensuremath{\textup{Out}}}

\title
[Approximate Cone Factorizations and Lifts of Polytopes]
{Approximate Cone Factorizations \\ and Lifts of Polytopes}

\author{Jo{\~a}o Gouveia}
\address{CMUC, Department of Mathematics,
  University of Coimbra, 3001-454 Coimbra, Portugal}
\email{jgouveia@mat.uc.pt} 

\author{Pablo A. Parrilo}
\address{Department of Electrical Engineering and Computer Science,
  Laboratory for Information and Decision Systems, Massachusetts
  Institute of Technology, 77 Massachusetts Avenue, Cambridge, MA
  02139, USA} 
\email{parrilo@mit.edu} 

\author{Rekha R. Thomas}
\address{Department of Mathematics, University of Washington, Box
  354350, Seattle, WA 98195, USA} \email{rrthomas@uw.edu}

\thanks{Gouveia was supported by the Centre for Mathematics at the
  University of Coimbra and Funda\c{c}\~ao para a Ci\^encia e a
  Tecnologia, through the European program COMPETE/FEDER. Parrilo was
  supported by AFOSR FA9550-11-1-0305, and Thomas by the U.S. National
  Science Foundation grant DMS-1115293.  }

\date{\today}
\begin{document}
\maketitle

\begin{abstract} 
In this paper we show how to construct inner and outer convex
approximations of a polytope from an approximate cone factorization
of its slack matrix.  This provides a robust generalization of the
famous result of Yannakakis that polyhedral lifts of a polytope are
controlled by (exact) nonnegative factorizations of its slack matrix.
Our approximations behave well under polarity and have efficient
representations using second order cones. We establish a direct
relationship between the quality of the factorization and the quality
of the approximations, and our results extend to generalized slack
matrices that arise from a polytope contained in a polyhedron.
\end{abstract}

\section{Introduction}
A well-known idea in optimization to represent a complicated convex
set $C \subset \RR^n$ is to describe it as the linear image of a
simpler convex set in a higher dimensional space, called a {\em lift}
or {\em extended formulation} of $C$ . The standard way to express
such a lift is as an affine slice of some closed convex cone $K$,
called a $K$-{\em lift} of $C$, and the usual examples of $K$ are
nonnegative orthants $\RR^m_+$ and the cones of real symmetric
positive semidefinite matrices $\PSD^m$.  More precisely, $C$ has a $K$-lift,
where $K \subset \RR^m$, if there exists an affine subspace $L \subset
\RR^m$ and a linear map $\pi \,:\, \RR^m \rightarrow \RR^n$ such that
$C = \pi(K \cap L)$.

Given a nonnegative matrix 
$M \in \RR^{p \times q}_+$ and a closed convex cone 
$K \subset \RR^m$ with dual cone $K^\ast \subset (\RR^m)^\ast$, a 
$K$-factorization of $M$ is a collection of elements 
$a_1, \ldots, a_p \in K^\ast$ and $b_1, \ldots, b_q \in K$ such that 
$M_{ij} = \langle a_i, b_j \rangle$ for all $i,j$.  In particular, a 
$\RR^m_+$-factorization of $M$, also called a {\em nonnegative factorization} 
of $M$ of size $m$, is typically expressed as $M = A^TB$ where 
$A$ has columns $a_1, \ldots, a_p \in (\RR^m)^\ast_+$ and $B$ 
has columns $b_1, \ldots, b_q \in \RR^m_+$. In \cite{Yannakakis}, Yannakakis
laid the foundations of polyhedral lifts of polytopes by showing the following.

\begin{theorem} \label{thm:Yannakakis} \cite{Yannakakis}
A polytope $P \subset \RR^n$ has a $\RR^m_+$-lift if and only if the
{\em slack matrix} of $P$ has a $\RR^m_+$-factorization. 
\end{theorem}

This theorem was extended in 
\cite{GPT2012} from $\RR^m_+$-lifts of polytopes to $K$-lifts of convex 
sets $C \subset \RR^n$, where $K$ is any closed convex cone,  
via $K$-{\em factorizations} of the {\em slack operator} of $C$.

The above results rely on exact cone factorizations of the slack
matrix or operator of the given convex set, and do not offer any
suggestions for constructing lifts of the set in the absence of exact
factorizations. In many cases, one only has access to approximate
factorizations of the slack matrix, typically via numerical
algorithms. In this paper we show how to take an approximate
$K$-factorization of the slack matrix of a polytope and construct from
it an inner and outer convex approximation of the polytope. Our
approximations behave well under polarity and admit efficient
representations via second order cones.  Further, we show that the
quality of our approximations can be bounded by the error in the corresponding 
approximate factorization 
of the slack matrix.

Let $P := \{x \in \RR^n \,:\, H^T x \leq {\mathbbm 1}\}$ be a
full-dimensional polytope in $\RR^n$ with the origin in its interior,
and vertices $p_1, \ldots, p_v$.  We may assume without loss of
generality that each inequality $\langle h_i, x \rangle \leq 1$ in
$H^Tx \leq {\mathbbm 1}$ defines a facet of $P$. If $H$ has size $n
\times f$, then the {\em slack matrix} of $P$ is the $f \times v$
nonnegative matrix $S$ whose $(i,j)$-entry is $S_{ij} = 1 - \langle
h_i, p_j \rangle$, the {\em slack} of the $j$th vertex in the $i$th
inequality of $P$.  Given an $\RR^m_+$-factorization of $S$, i.e., two
nonnegative matrices $A$ and $B$ such that $S = A^TB$, an
$\RR^m_+$-lift of $P$ is obtained as
\[
P=\left\{ x \in \RR^n \, : \, \exists y \in \RR^{m}_+ \textup{ s.t. }  H^Tx + A^T y = {\mathbbm 1} \right\}.
\]
Notice that this lift is
highly non-robust, and small perturbations of $A$ make the right hand
side empty, since the linear system $H^Tx + A^T y = {\mathbbm 1}$ is in general
highly overdetermined. The same sensitivity holds for all
$K$-factorizations and lifts.  Hence, it becomes important to have a
more robust, yet still efficient, way of expressing $P$ (at least
approximately) from approximate $K$-factorizations of $S$. Also, the
quality of the approximations of $P$ and their lifts must reflect the
quality of the factorization, and specialize to the Yannakakis setting
when the factorization is exact. The results in this paper carry out
this program and contain several examples, special cases, and
connections to the recent literature.

\subsection{Organization of the paper} In Section~\ref{sec:approx} we
establish how an approximate $K$-factorization of the slack matrix of
a polytope $P \subset \RR^n$ yields a pair of inner and outer convex
approximations of $P$ which we denote as $\Inn_P(A)$ and $\Out_P(B)$
where $A$ and $B$ are the two ``factors'' in the approximate
$K$-factorization. These convex sets arise naturally from two simple
inner and outer second order cone approximations of the nonnegative
orthant. While the outer approximation is always closed, the inner 
approximation maybe open if $K$ is an arbitrary cone. However, we show 
that if the polar of $K$ is  ``nice'' \cite{Pataki1}, then the inner approximation will be closed.
All cones of interest to us in this paper such as nonnegative orthants, 
positive semidefinite cones, and second order cones are nice. Therefore, we will 
assume that our approximations are closed after a discussion of their closedness.

We prove that our approximations behave well under
polarity, in the sense that
\[
\Out_{P^{\circ}}(A) = (\Inn_{P}(A))^{\circ} \quad \textup{ and } \quad
\Inn_{P^{\circ}}(B) = (\Out_{P}(B))^{\circ}
\] 
where $P^\circ$ is the polar polytope of $P$. Given $P \subset \RR^n$
and $K \subset \RR^m$, our approximations admit efficient
representations via slices and projections of $K \times \SOC_{n+m+2}$ where
$\SOC_k$ is a second order cone of dimension $k$. We show that an
$\varepsilon$-error in the $K$-factorization makes $\frac{1}{1 +
  \varepsilon} P \subseteq \Inn_P(A)$ and $\Out_P(B) \subseteq (1+
\varepsilon)P$, thus establishing a simple link between the error in
the factorization and the gap between $P$ and its approximations.  In
the presence of an exact $K$-factorization of the slack matrix, our
results specialize to the Yannakakis setting.

In Section~\ref{sec:cases} we discuss two connections between our
approximations and well-known constructions in the literature. In the
first part we show that our inner approximation, $\Inn_P(A)$, always contains
the Dikin ellipsoid used in interior point methods. Next we examine
the closest rank one approximation of the slack matrix obtained via a
singular value decomposition and the approximations of the polytope
produced by it.

In Section~\ref{sec:twomatrices} we extend our results to the case of
generalized slack matrices that arise from a polytope contained in a
polyhedron.  We also show how an approximation of $P$ with a $K$-lift
produces an approximate $K$-factorization of the slack matrix of
$P$. It was shown in \cite{BraunFioriniPokuttaSteurer} that the max
clique problem does not admit polyhedral approximations with small
polyhedral lifts.  We show that this negative result continues to hold
even for the larger class of convex approximations considered in this
paper.

\section{From approximate factorizations to approximate lifts} 
\label{sec:approx}

In this section we show how to construct inner and outer
approximations of a polytope~$P$ from approximate $K$-factorizations
of the slack matrix of $P$, and establish the basic properties of
these approximations.

\subsection{$K$-factorizations and linear maps}
Let $P = \{ x \in \RR^n \,:\, H^T x \leq {\mathbbm 1} \}$ be a
full-dimensional polytope with the origin in its interior. The
vertices of the polytope are $p_1, \ldots, p_v$, and each inequality
$\langle h_i, x \rangle \leq 1$ for $i=1,\ldots, f$ in $H^Tx \leq
{\mathbbm 1}$ defines a facet of $P$. The \emph{slack matrix} $S$ of
$P$ is the $f \times v$ matrix with entries $S_{ij} = 1 - \langle h_i,
p_j \rangle$. In matrix form, letting $H = [h_1 \ldots h_f]$ and $V =
[p_1 \ldots p_v]$, we have the expression $S = {\mathbbm 1}_{f \times
  v} - H^T V$. We assume $K \subset \RR^m$ is a closed convex cone,
with dual cone $K^\ast = \{ y \in (\RR^m)^\ast \, : \, \langle
y,x\rangle \geq 0 \quad \forall x \in K\}$.

\begin{definition}{(\cite{GPT2012})}
\label{def:kfact}
A \emph{$K$-factorization} of the slack matrix $S$ of the polytope $P$
is given by $a_1, \ldots, a_f \in K^\ast$, $b_1, \ldots, b_v \in K$
such that $1 - \langle h_i , p_j \rangle = \langle a_i , b_j \rangle$
for $i = 1,\ldots,f$ and $j=1,\ldots,v$.  In matrix form, this is the
factorization
\[ 
S = {\mathbbm 1}_{f \times v} - H^T V = A^TB
\]
where $A = [a_1 \ldots a_f]$ and $B =[b_1 \ldots b_v]$.
\end{definition}
It is convenient to interpret a $K$-factorization as a composition of
linear maps as follows.  Consider $B$ as a linear map from $\RR^v
\rightarrow \RR^m$, verifying $B(\RR^v_+) \subseteq K$. Similarly,
think of $A$ as a linear map from $(\RR^f)^\ast \rightarrow
(\RR^m)^\ast$ verifying $A((\RR^f)^\ast_+)\subseteq K^\ast$.  Then,
for the adjoint operators, $B^\ast(K^\ast)\subseteq (\RR^v)^\ast_+$
and $A^\ast(K) \subseteq \RR^f_+$. Furthermore, we can think of the
slack matrix $S$ as an affine map from $\RR^v$ to $\RR^f$, and the
matrix factorization in Definition~\ref{def:kfact} suggests to define
the slack operator, $S:\RR^v \rightarrow \RR^f$, as $S(x)=({\mathbbm
  1}_{f \times v} - H^\ast \circ V)(x)$, where $V:\RR^v \rightarrow
\RR^n$ and $H:(\RR^f)^\ast \rightarrow (\RR^n)^\ast$.
\[
\xymatrix{ & &  K \ar@{}|{\rotatebox[origin=c]{270}{$\subseteq$}}[d]& & \\ 
& & {\RR^m}  \ar^{A^*}[ddrr] \ar@{..>}^{ \scriptsize \begin{pmatrix} 1 \\ \pi \end{pmatrix}}[dd]& & \\ \\ 
\RR^v \ar^{B}[uurr] \ar_{\scriptsize \begin{bmatrix}\mathbbm{1}_v^*\\V \end{bmatrix}}[rr]&  & 
\RR \oplus \RR^n \ar_{\scriptsize \begin{bmatrix} \mathbbm{1}_f & -H^* \end{bmatrix}}[rr]& & \RR^f }
\]
We define a {\em nonnegative $K$-map} from $\RR^v \rightarrow \RR^m$
(where $K\subseteq \RR^m$) to be any linear map $F$ such that
$F(\RR_+^v) \subseteq K$. In other words, a nonnegative $K$-map from
$\RR^v \rightarrow \RR^m$ is the linear map induced by an assignment
of an element $b_i \in K$ to each unit vector $e_i \in \RR^v_+$.  In
this language, a $K$-factorization of $S$ corresponds to a nonnegative
$K^\ast$-map $A \,:\, (\RR^f)^\ast \rightarrow (\RR^m)^\ast$ and a
nonnegative $K$-map $B \,:\, \RR^v \rightarrow \RR^m$ such that $S(x)
= (A^\ast \circ B)(x)$ for all $x\in \RR^v$. As a consequence, we have
the correspondence $a_i := A(e_i^*)$ for $i=1,\ldots,f$ and $b_j :=
B(e_j)$ for $j=1,\ldots,v$.

\subsection{Approximations of the nonnegative orthant} 
In this section we introduce two canonical second order cone
approximations to the nonnegative orthant, which will play a crucial
role in our developments. In what follows, $\|\cdot\|$ will always
denote the standard Euclidean norm in $\RR^n$, i.e., $\|x\| =
(\sum_{i=1}^n x_i^2)^\frac{1}{2}$.
\begin{definition}
Let $\kin^n$, $\kout^n$ be the cones
\begin{align*}
\kin^n &:= \{ x \in \RR^n \, : \, \sqrt{n-1} \cdot \|x\| \leq
\mathbbm{1}^T x \}\\ 
\kout^n&:= \{ x \in \RR^n \, : \, \|x\| \leq
\mathbbm{1}^T x \}.
\end{align*}
If the dimension $n$ is unimportant or obvious from the context, we
may drop the superscript and just refer to them as $\kin$ and
$\kout$.
\end{definition}
As the following lemma shows, the cones $\kin$ and $\kout$ provide
inner and outer approximations of the nonnegative orthant, which are
dual to each other and can be described using second-order cone
programming (\cite{AlizadehGoldfarb,Lobo}).
\begin{lemma}
The cones $\kin$ and $\kout$ are proper cones (i.e., convex,
closed, pointed and solid) in $\RR^n$ that satisfy
\[
\kin \subseteq \RR^n_+ \subseteq \kout, 
\]
and furthermore, $\kin^* = \kout$, and $\kout^* = \kin$.
\label{lem:kinkout}
\end{lemma}
The cones $\kin$ and $\kout$ are in fact the ``best'' second-order
cone approximations of the nonnegative orthant, in the sense that they
are the largest/smallest permutation-invariant cones with these
containment properties; see also
Remark~\ref{rem:renegar}. Lemma~\ref{lem:kinkout} is a direct
consequence of the following more general result about (scaled)
second-order cones:
\begin{lemma}
Given $\omega \in \RR^n$ with $\|\omega\|=1$ and $0 < a < 1$, consider
the set
\[
K_{a} := \{ x \in \RR^n \, : \, a \|x\| \leq \omega^T x \}.
\]
Then, $K_a$ is a proper cone, and $K_a^* = K_b$, where $b$ satisfies
$a^2 + b^2 =1$, $b > 0$.
\label{lem:kakb}
\end{lemma}
\begin{proof}
The set $K_a$ is clearly invariant under nonnegative scalings, so it
is a cone.  Closedness and convexity of $K_a$ follow directly from the
fact that (for $a \geq 0$) the function $x \mapsto a \|x\| - \omega^T
x$ is convex. The vector $\omega$ is an interior point (since $a
\|\omega\| - \omega^T \omega = a - 1 < 0$), and thus $K_a$ is
solid. For pointedness, notice that if both $x$ and $-x$ are in $K_a$,
then adding the corresponding inequalities we obtain $2 a \|x\| \leq
0$, and thus (since $a>0$) it follows that $x=0$.
 
The duality statement $K_a^* = K_b$ is perhaps geometrically obvious,
since $K_a$ and $K_b$ are spherical cones with ``center'' $\omega$ and
half-angles $\theta_a$ and $\theta_b$, respectively, with $\cos \theta_a =
a$, $\cos \theta_b =b$, and $\theta_a+\theta_b= \pi/2$. For completeness,
however, a proof follows.  We first prove that $K_b \subseteq
K_a^*$. Consider $x \in K_a$ and $y \in K_b$, which we take to have
unit norm without loss of generality. Let $\alpha,\beta,\gamma$ be
the angles between $(x,\omega)$, $(y,\omega)$ and $(x,y)$,
respectively. The triangle inequality in spherical geometry gives
$\gamma \leq \alpha + \beta$. Then,
\[
\cos \gamma \geq \cos(\alpha+\beta) =
\cos \alpha \cos \beta - \sin \alpha \sin \beta 
\]
or equivalently
\begin{align*}
x^T y & \geq (\omega^T x) (\omega^T y) - \sqrt{1-(\omega^T x)^2} \sqrt{1-(\omega^T y)^2}
\geq a b - \sqrt{1-a^2}\sqrt{1-b^2} = 0.
\end{align*}
To prove the other direction ($K_a^* \subseteq K_b$) we use its
contrapositive, and show that if $y \not \in K_b$, then $y \not \in
K_a^*$. Concretely, given a $y$ (of unit norm) such that $b \|y\| >
\omega^T y$, we will construct an $x \in K_a$ such that $y^T x <0$
(and thus, $y \not \in K_a^*$). For this, define $x := a \omega - b
\hat \omega$, where $\hat \omega := (y - (\omega^T y)
\omega)/\sqrt{1-(\omega^T y)^2}$ (notice that $\|\hat \omega\| =1$ and
$\omega^T \hat \omega =0$).  It can be easily verified that $\omega^T
x = a$ and $\|x\|^2= a^2+b^2=1$, and thus $x \in K_a$.  However, we have
\[
y^T x = a (\omega^T y) - b \sqrt{1 - (\omega^T y)^2} < ab - b \sqrt{1-b^2} = 0,
\]
which proves that $y \not \in K_a^*$.
\end{proof}

\begin{proof}[of Lemma~\ref{lem:kinkout}]
Choosing $\omega = \mathbbm{1}/\sqrt{n}$, $a = \sqrt{(n-1) / n}$ and
$b = \sqrt{1/n}$ in Lemma~\ref{lem:kakb}, we have $\kin = K_a$ and
$\kout = K_b$, so the duality statement follows. Since $x = \sum_i x_i
e_i$, with $x_i \geq 0$, we have
\[
\|x \| \leq \sum_i x_i \|e_i\| = \sum_i x_i = \mathbbm{1}^T x, 
\]
and thus $\RR_+^n \subseteq \kout$. Dualizing this expression, and
using self-duality of the nonnegative orthant, we obtain the remaining
containment $\kin = \kout^* \subseteq (\RR^n_+)^* = \RR^n_+$.
\end{proof}

\begin{remark}
\label{rem:renegar}
Notice that $(\mathbbm{1}^T x)^2 - \|x\|^2 = 2 \sigma_2(x)$, where $\sigma_2(x)$
is the second elementary symmetric function in the variables
$x_1,\ldots,x_n$. Thus, the containment relations in
Lemma~\ref{lem:kinkout} also follow directly from the fact that the
cone $\kout$ is the $(n-2)$ derivative cone (or \emph{Renegar
  derivative}) of the nonnegative orthant; see e.g. \cite{hypRenegar}
for background and definitions and \cite{SaundersonParrilo} for their
semidefinite representability.
\end{remark}

\begin{remark}
\label{rem:Kinalt}
The following alternative description of $\kin$ is often
convenient:
\begin{equation}
\kin = \{ x \in \RR^n \, : \, \exists t \in \RR \quad \mbox{ s.t. } \|t {\mathbbm 1} - x \| \leq t \}. 
\label{eq:Kinalt}
\end{equation}
The equivalence is easy to see, since the condition above requires the
existence of $t \geq 0$ such that $t^2 (n-1) - 2 t {\mathbbm 1}^T x +
\|x\|^2 \leq 0$. Eliminating the variable $t$ immediately yields
$\sqrt{n-1} \cdot \|x\| \leq {\mathbbm 1}^T x$. The containment
$\kin \subseteq \RR^n_+$ is now obvious from this representation,
since $t - x_i \leq \| t {\mathbbm 1} -x\| \leq t$, and thus $x_i \geq
0$.
\end{remark}

\subsection{From orthants to polytopes} 
The cones $\kin$ and $\kout$ provide ``simple'' approximations to the
nonnegative orthant. As we will see next, we can leverage these to
produce inner/outer approximations of a polytope from any approximate
factorizations of its slack matrix. The constructions below will use
\emph{arbitrary} nonnegative $K^\ast$ and $K$-maps $A$ and $B$ (of suitable
dimensions) to produce approximations of the polytope $P$ (though of
course, for these approximations to be useful, further conditions will
be required).

\begin{definition}
\label{def:inout}
Given a polytope $P$ as before, a nonnegative $K^\ast$-map
$A:(\RR^f)^\ast \rightarrow (\RR^m)^\ast$ and a nonnegative $K$-map
$B:\RR^v\rightarrow \RR^m$ , we define the following two sets:
\begin{align*}
\Inn_{P}(A) &:= \left\{ x \in \RR^n \, : \, \exists y \in K, \, \textup{ s.t. }  {\mathbbm 1}-H^\ast(x) - A^{\ast}(y) \in \kin^f \right\},\\
\Out_{P}(B) &:= \left\{ V(z) \, : \, z \in \kout^v, \quad {\mathbbm 1}^T z \leq 1, \quad B(z) \in K \right\}.
\end{align*}
\end{definition}
By construction, these sets are convex, and the first observation is
that the notation makes sense as the sets indeed define an inner and
an outer approximation of $P$.

\begin{proposition} \label{prop:inclusions}
Let $P$ be a polytope as before and $A$ and $B$ be nonnegative
$K^\ast$ and $K$-maps respectively. Then $\Inn_{P}(A) \subseteq P
\subseteq \Out_{P}(B)$.
\end{proposition}

\begin{proof}
If $x \in \Inn_{P}(A)$, there exists $y \in K$ such that ${\mathbbm 1} - H^\ast(x) - A^{\ast}(y) \in
\kin^f \subseteq \RR^f_+$, which implies $H^\ast(x) + A^{\ast}(y) \leq
{\mathbbm 1}$. Since $A^\ast(y) \geq 0$ for $y \in K$, we have $H^\ast(x) \leq {\mathbbm
  1}$ and thus $x \in P$.
  
 For the second inclusion, by the convexity of $\Out_{P}(B)$ it is
 enough to show that the vertices of $P$ belong to this set.  Any
 vertex $p$ of $P$ can be written as $p=V(e_i)$ for some canonical
 basis vector $e_i$. Furthermore, $B(e_i) \in K$ since $B$ is a
 nonnegative $K$-map, $e_i \in \RR^v_+ \subseteq \kout^v$,
 ${\mathbbm 1}^T e_i \leq 1$, and so $p \in \Out_P(B)$ as intended.
\end{proof}

If $A$ and $B$ came from a true $K$-factorization of $S$, then $P$ has
a $K$-lift and $S = A^\ast \circ B$. Then, the subset of $\Inn_P(A)$ given by
$\left\{ x \in \RR^n \,:\, \exists y \in K \textup{ s.t. } H^\ast(x) +
A^{\ast}(y) = {\mathbbm 1} \right\}$ contains $P$, since it contains
every vertex $p_i = V(e_i)$ of $P$. This can be seen by taking $y =
B(e_i) \in K$ and checking that ${\mathbbm 1} - H^\ast(p_i) =
({\mathbbm 1}_{f \times v} - H^\ast \circ V)(e_i) = (A^\ast \circ B) (e_i)
= A^\ast(y)$.  From Proposition~\ref{prop:inclusions} it then follows
that $\Inn_P(A) = P$. The definition of $\Out_P(B)$ can be similarly
motivated. An alternative derivation is through polarity, as we will
see in Theorem~\ref{thm:polarity}.

\begin{remark} \label{rem:alt Out}
Since the origin is in $P$, the set $\Out_P(B)$ can also be defined with the inequality ${\mathbbm 1}^Tz \leq 1$ replaced by the corresponding equation:
$$\Out_P(B) = \left\{ V(z) \, : \, z \in \kout^v, \quad {\mathbbm 1}^T z = 1, \quad B(z) \in K \right\}.$$
To see this, suppose $q := V(z)$ such that $z \in \kout^v$, ${\mathbbm 1}^T z \leq 1$ and $B(z) \in K$. 
Then there exists $s \geq 0$ such that ${\mathbbm 1}^T z + s = 1$. Since $0 \in P$, there exists $0 \leq \lambda_i $ with $\sum \lambda_i =1$ such that $0 = \sum \lambda_i p_i$ where $p_i$ are the vertices of $P$.  
Let $\tilde{z} := s \lambda + z \in \RR^v$ where $\lambda = (\lambda_i)$. 
Then $\tilde{z} \in \kout^v$ since $s \geq 0$, $\lambda \in \RR^v_+ \subseteq \kout^v$ and $z \in \kout^v$. 
Further, ${\mathbbm 1}^T \tilde{z} = s {\mathbbm 1}^T \lambda + {\mathbbm 1}^Tz = 1$. We also have that $B(\tilde{z}) = B(s \lambda + z)  = sB(\lambda) + B(z) \in K$ since each component is in $K$ (note that $B(\lambda) \in K$ since $\lambda \in \RR^v_+$). 
Therefore, we can write $q =V(\tilde{z})$ with $\tilde{z}\in \kout^v$, ${\mathbbm 1}^T \tilde{z} = 1$ and $B(\tilde{z}) \in K$
which proves our claim.  This alternate formulation of $\Out_P(B)$ will be useful in Section 4. However, Definition~\ref{def:inout}
is more natural for the polarity results in Section~\ref{subsec:polarity}.
\end{remark}

\begin{example}\label{Ex:zeroapprox}
Let $P$ be the $n$-dimensional simplex given by the inequalities:
$$P = \left\{ x \in \RR^n \,:\ 1+x_1 \geq 0, \; \ldots, \; 1+x_n \geq 0, \; 1- \textstyle{\sum}_i x_i \geq 0 \right\}$$ 
with vertices $(n,-1,\cdots,-1), (-1,n,-1,\ldots,-1), \ldots, (-1,\ldots,-1,n),(-1,\ldots,-1)$. The slack matrix of this polytope is 
the $(n+1) \times (n+1)$ diagonal matrix with all diagonal entries equal to $n+1$. Choosing $A$ to be the zero map, for any cone $K$ we have 
\begin{align*}
%\Inn_{P}(0) & = & \left\{ x \in \RR^n \, : \, \exists \eta \in \RR \textup{ s.t. } \left\|  \begin{array}{c}
%-x_1 - \eta \\ \vdots \\ -x_n-\eta \\ \sum x_i - \eta  \end{array} \right\|_2 
%\leq 1-\eta \right\}\\
\Inn_{P}(0) & = \left\{ x \in \RR^n \, : \, n \left(\sum_i
(1+x_i)^2+(1-\sum_i x_i)^2 \right) \leq (n+1)^2 \right\}\\ & = \left\{
x \in \RR^n \, : \, \sum_i x_i^2 + \left(\sum_i x_i \right)^2 \leq
\frac{(n+1)}{n}\right \}.
\end{align*}
%Since $1-2 \eta - n \eta^2$ is maximized at $\eta = \frac{-1}{n}$, we get that 
%$$\Inn_{P}(0)  = \left\{ x \in \RR^n \, : \, \sum x_i^2 + \left(\sum x_i \right)^2 \leq \frac{(n+1)}{n}\right \}.$$
For the case of $n=2$ we have:
\[
\Inn_P(0) =  \left\{ (x_1, x_2) \,:\, x_1^2 + x_2^2 + (x_1+x_2)^2 \leq \frac{3}{2} \right\} 
= \left\{  (x_1,x_2) \,:\, 3(x_1+x_2)^2 + (x_1-x_2)^2 \leq 3 \right\}.
\]

For the outer approximation, if we choose $B=0$ then we obtain the body 
\[
\Out_P(0) = \left\{ \left( -\sum_{i=1}^{n+1} z_i + (n+1)z_j,  \,\,\,j=1,\ldots,n \right), \,\,\,\|z\| \leq \sum_{i=1}^{n+1} z_i \leq 1 \right\}.
\]

For $n=2$, $\Out_P(0) = \{ (2z_1-z_2-z_3, -z_1+2z_2-z_3) \,:\, \|z\| \leq z_1+z_2+z_3 \leq 1 \},$ and eliminating variables we get
$$\Out_P(0) =  \left\{  (x_1,x_2) \,:\, 3(x_1+x_2)^2 + (x_1-x_2)^2 \leq 12 \right\}.$$
The simplex and its approximations can be seen in Figure~\ref{fig:zeroapprox1}.

Note that the bodies $\Inn_P(0)$ and $\Out_P(0)$ do not depend on the
choice of a cone $K$ and are hence canonical convex sets associated to
the given representation of the polytope $P$. However, while
$\Out_P(0)$ is invariant under translations of $P$ (provided the
origin remains in the interior), $\Inn_P(0)$ is sensitive to
translation, i.e., to the position of the origin in the polytope
$P$. To illustrate this, we translate the simplex in the above example
by adding $(\frac{1}{2}, \frac{1}{2}, \ldots, \frac{1}{2})$ to it and
denote the resulting simplex by $Q$. Then $$Q = \left\{ x \in \RR^n
\,:\, 1+2x_1 \geq 0, \ldots, 1+2x_n \geq 0,\,\, 1 - \sum
\frac{2}{n+2}x_i \geq 0 \right\}$$ and its vertices are
$(n+\frac{1}{2},-\frac{1}{2},-\frac{1}{2},\ldots,-\frac{1}{2}),
\ldots, ( -\frac{1}{2},-\frac{1}{2},\ldots,
-\frac{1}{2},n+\frac{1}{2}), (-\frac{1}{2}, -\frac{1}{2}, \ldots,
-\frac{1}{2})$.

\begin{figure} 
\centering
\includegraphics[scale=0.2]{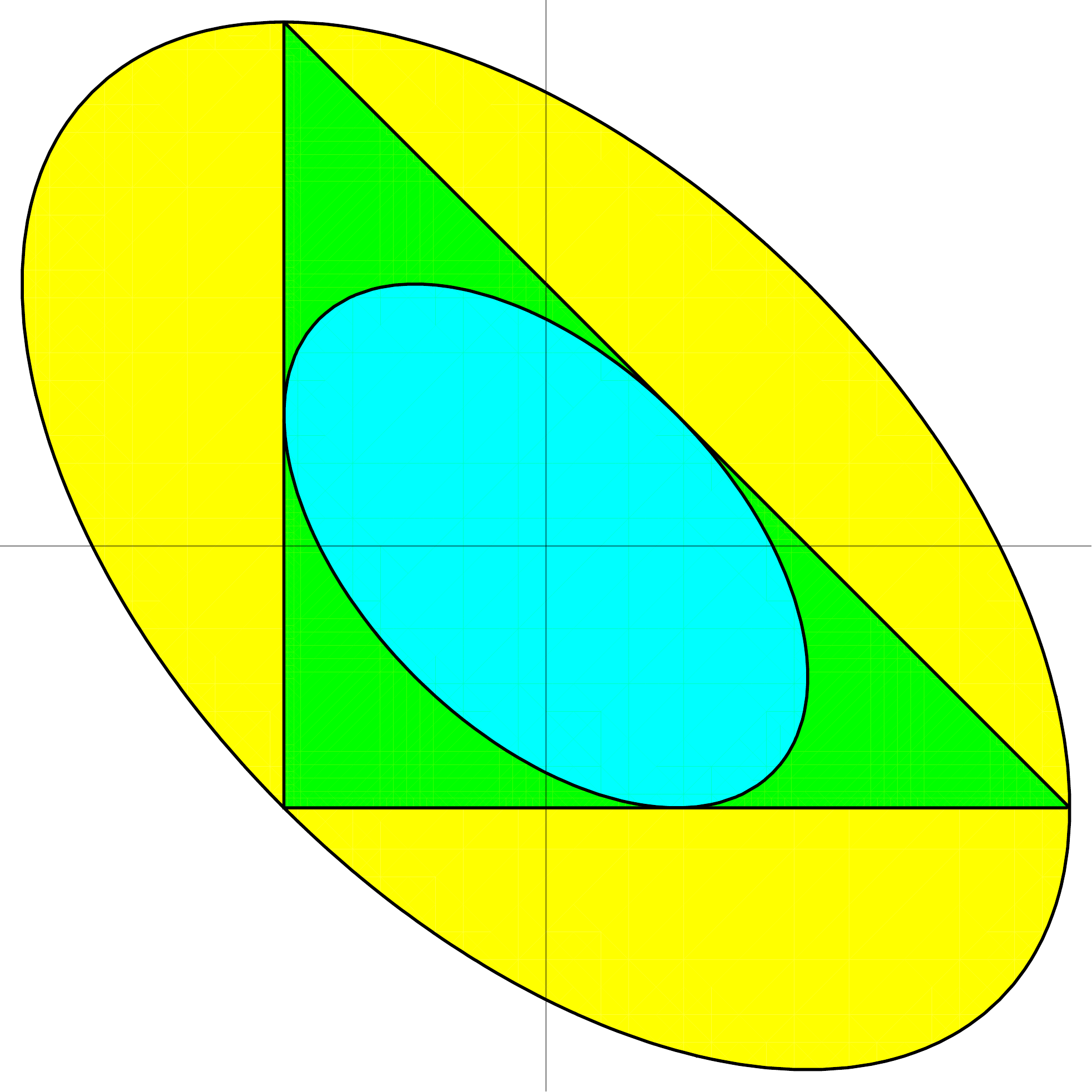}
\caption{$\Inn_{P}(0)$ and $\Out_P(0)$ for a triangle centered at the origin.}
\label{fig:zeroapprox1}
\end{figure}

For $n=2$, plugging into the formula for the inner approximation we
get $$\Inn_Q(0)=\{(x,y):(3(x+y)-2)^2+16(x-y)^2 \leq 16\},$$ while
doing it for the outer approximation
yields $$\Out_Q(0)=\{(x,y):3(x+y-1)^2+(x-y)^2 \leq 12\}.$$ So we can
see that while $\Out_Q(0)$ is simply a translation of the previous
one, $\Inn_Q(0)$ has changed considerably as can be seen in Figure~\ref{fig:zeroapprox2}.

\begin{figure} 
\centering
\includegraphics[scale=0.25]{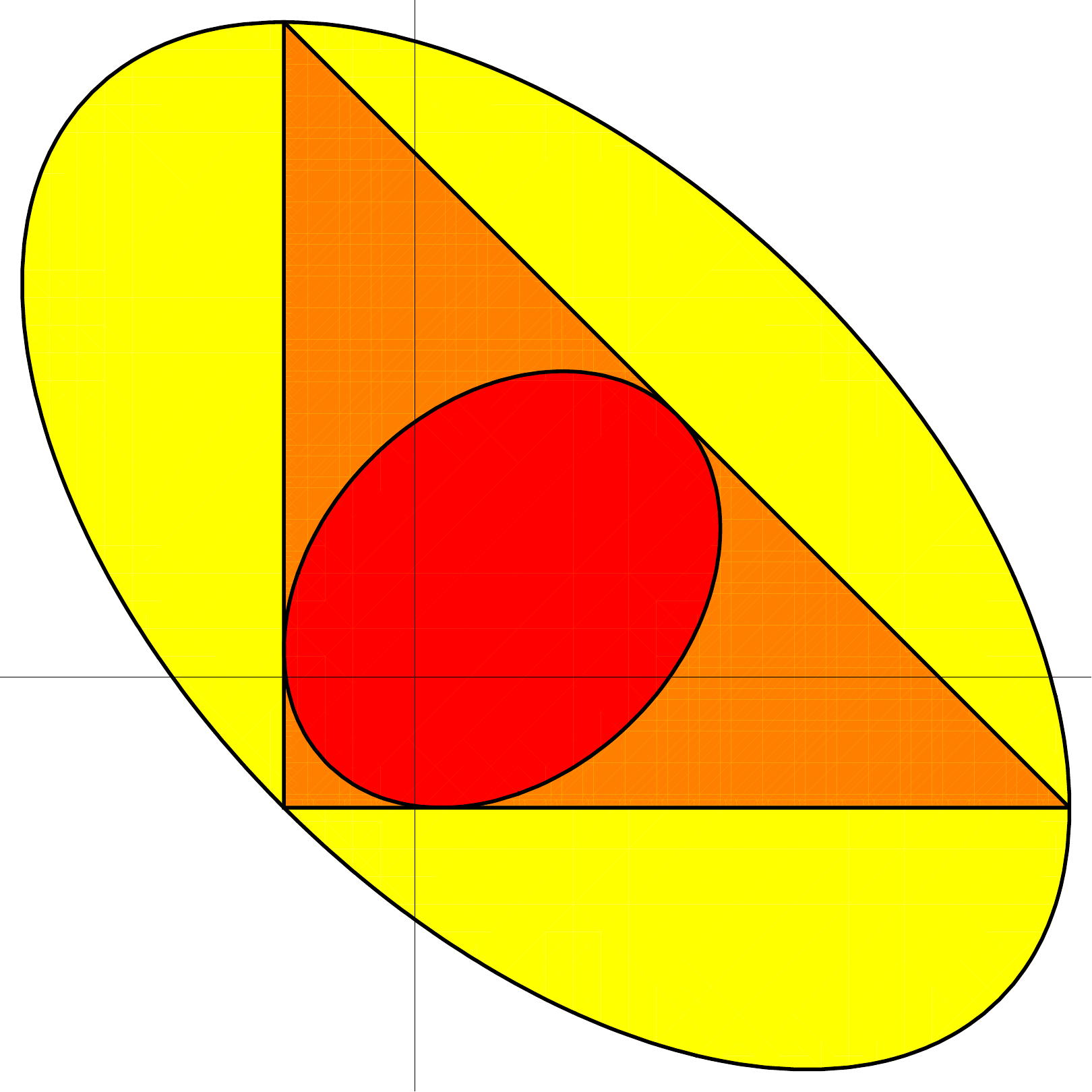}
\caption{$\Inn_{Q}(0)$ and $\Out_Q(0)$ for a triangle not centered at the origin.}
\label{fig:zeroapprox2}
\end{figure}
\end{example}

\subsection{Closedness of $\Inn_P(A)$}
It is easy to see that $\Out_P(B)$
is closed and bounded. Indeed, since ${\mathbbm 1}$ lies in the interior of $\kin^v$ which is the polar of $\kout^v$, for every $z \in  \kout^v$, ${\mathbbm 1}^T z > 0$. Therefore, there is no $z \in \kout^v$ such that ${\mathbbm 1}^T (\lambda z) = \lambda ({\mathbbm 1}^T z) \leq 1$ for all $\lambda > 0$ and so, $\{ z \in \kout^v \,:\, {\mathbbm 1}^T z \leq 1 \}$ is bounded. It is closed since $\kout^v$ is closed.
Further, $\{ z \in \kout^v \,:\, {\mathbbm 1}^T z \leq 1, \, B(z) \in K \}$ is also compact since $B^{-1}(K) \subset \RR^v$ is closed as $B$ is a linear map and $K$ is closed. Therefore, $\Out_P(B)$ is compact since it is the linear image of a compact set.
The set $\Inn_P(A)$ is bounded since it is contained in the polytope $P$. We will also see in Proposition~\ref{prop:colcone error bound} that it has an interior. However, $\Inn_P(A)$ may not be closed.
 
\begin{example}Consider the three dimensional cone
$$K=\left\{(a,b,t) \in \RR^3 \ : \ \begin{pmatrix} a+b & 0 & 2(a-b)\\0 & t & 2(a+b)-t\\2(a-b) & 2(a+b)-t & t \end{pmatrix} \succeq 0\right\},$$
and take $A^*$ to be the map sending $(a,b,t)$ to $(a,b,0)$. Then, $$A^*(K)=\{(x,y,0) \ : x>0, y>0\} \cup \{(0,0,0)\}.$$ For any triangle $P \subset \RR^2$ given by 
${ 1}- H^*(x) \geq 0$ we then have that 
$$\Inn_{P}(A)=\left\{ x \in \RR^2 \, : \,  {1}-H^*(x) \in \kin^3 + A^*(K) \right\}.$$
Since $\kin^3$ is a second order cone that is strictly contained in $\RR^3_+$ and 
$A^*(K)$ has the description given above, the cone $\kin^3 + A^*(K)$ is not closed.
The set $\Inn_{P}(A)$ is an affine slice of this non-closed cone and therefore, may not be closed.
Taking, for example, the triangle 
$$P = \left\{ (x,y) \in \RR^2 \,:\, \begin{pmatrix} 1-x \\ 1-y \\ 1+x+y \end{pmatrix} \geq 0 \right\},$$
$$\Inn_{P}(A)=\left\{ (x,y) \in \RR^2 \, : \,  (1-x,1-y,1+x+y) \in \kin^3 + A^*(K) \right\},$$
which is not closed, as can be seen in Figure~\ref{fig:nonclosed}.
\begin{figure} 
\centering
\includegraphics[scale=0.3]{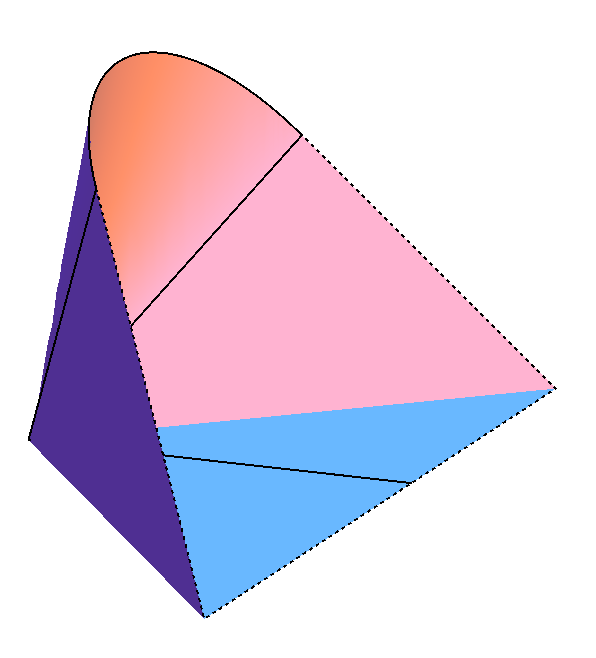} \hspace{1cm}
\includegraphics[scale=0.3]{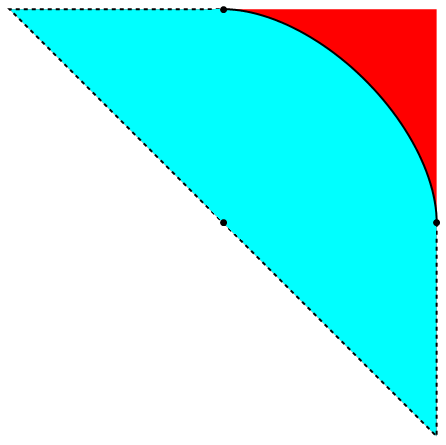}
\caption{Affine slice of $\kin^3 + A^*(K)$, and the corresponding projection $\Inn_{P}(A)$.}
\label{fig:nonclosed}
\end{figure}
Notice that in this example, $A$ is a nonnegative $K^\ast$-map, 
but $A^\ast(K)$ is not closed. 
\end{example}
 
\begin{remark} \label{rem:A*(K) closed}
If $A^\ast(K)$ is
closed, then $\kin^f + A^\ast(K)$ will be closed by \cite[Corollary
  9.12]{Rockafellar} since both cones are closed, and contained in
$\RR^f_+$, which means that there is no direction of recession of one
cone whose opposite is a direction of recession of the other cone. If
$\kin^f + A^\ast(K)$ is closed then $\Inn_P(A)$ will be closed since
it is an affine slice of $\kin^f + A^\ast(K)$.
\end{remark}

We now discuss results from \cite{Pataki1} and \cite{Pataki2} that give conditions under which $A^\ast(K)$ is closed. 
%By Remark~\ref{rem:A*(K) closed}, $\Inn_P(A)$ will be closed when these conditions hold.

\begin{definition} \cite[Definition 1.1]{Pataki1}, \cite{Borwein-Wolkowicz}
A closed convex cone $C$ is {\em nice} if $C^\ast + F^\perp$ is closed for all faces $F$ of $C$.
\end{definition}

Nonnegative orthants, positive semidefinite cones, and second order cones are all nice. Given a convex set $C$, let $\textup{ri }C$ denote the relative interior of $C$. For $x \in C$, define 
$$\textup{dir} (x,C) = \{ y \,:\, x + \varepsilon y \in C \textup{ for some } \varepsilon > 0 \}$$
and $\textup{cl}(\textup{dir} (x,C))$ denote the closure of $\textup{dir} (x,C)$. For example, if $C \subset \RR^m$ is a full-dimensional closed convex cone and $x$ is in the interior of $C$, then $\textup{dir} (x,C) = \RR^m$. If $x \in \textup{ri } F$ for a face $F$ of $C$, then $\textup{dir} (x,C)$ contains the linear span of $F$ and hence, 
% is the sum of the linear span of $F$ and the 
%interior of the tangent cone at $F$, and $\textup{cl}(\textup{dir} (x,C))$ is the tangent cone at $F$. Therefore, 
$\textup{cl}(\textup{dir} (x,C))\backslash \textup{dir} (x,C)$ does not contain the linear span of $F$. 

\begin{theorem} \cite[Theorem 1]{Pataki2} \label{thm:Pataki}
Let $M$ be a linear map, $C$ a nice cone and $x \in \textup{ri }(C \cap \textup{range}(M))$. Then $M^\ast C^\ast$ is closed if and only if $\textup{range}(M) \cap (\textup{cl}(\textup{dir} (x,C)) \backslash \textup{dir} (x,C)) = \emptyset$.
\end{theorem}

\begin{corollary} \label{cor:nice implies closed}
If $K^\ast$ is a nice cone and $A$ is a nonnegative $K^\ast$-map, then both $A^\ast(K)$ and $\Inn_P(A)$ are closed.
\end{corollary}

\begin{proof} Since $A$ is a nonnegative $K^\ast$-map, $A$ sends $e_i^\ast \in (\RR^f)^\ast$ to $a_i \in K^\ast$. Therefore, $\textup{range}(A)$ 
is contained in the linear span of a face of $K^\ast$. Let $F$ be a
minimal such face.  If $x \in \textup{ri }(K^\ast \cap \textup{range}(A))$, 
then $x$ lies in $\textup{ri } F$ which means that
$\textup{cl}(\textup{dir} (x, K^\ast)) \backslash \textup{dir}
(x,K^\ast)$ does not contain the linear span of $F$ and hence does not
contain $\textup{range}(A)$. This implies that $\textup{range}(A) \cap
(\textup{cl}(\textup{dir} (x, K^\ast)) \backslash \textup{dir}
(x,K^\ast)) = \emptyset$ and $A^\ast(K)$ is closed by
Theorem~\ref{thm:Pataki}. By Remark~\ref{rem:A*(K) closed}, it follows
that $\Inn_P(A)$ is closed.
\end{proof}

\begin{remark} \label{rem:assume inn closed}
By Corollary~\ref{cor:nice implies closed}, $\Inn_P(A)$ is closed when $K$ is a nonnegative orthant, psd cone or a 
second order cone. Since these are precisely the cones that are of interest to us in this paper, we will assume that $\Inn_P(A)$ is closed in the rest of this paper. While this assumption will be convenient in the proofs of several forthcoming results, 
we note that the results would still be true if we were to replace $\Inn_P(A)$ by its closure everywhere.
\end{remark}
  
\subsection{Polarity} \label{subsec:polarity}
We now show that our approximations behave nicely under
polarity. Recall that the polar of the polytope $P \subset \RR^n$ is
the polytope $P^\circ = \{ y \in (\RR^n)^\ast \,:\, \langle y,x
\rangle \leq 1, \,\,\forall \,\, x \in P \}$. The face lattices of $P$
and $P^\circ$ are anti-isomorphic. In particular, if $P = \{ x \in
\RR^n \,:\, H^Tx \leq {\mathbbm1} \}$ with vertices $p_1, \ldots, p_v$
as before, then the facet inequalities of $P^\circ$ are $\langle y,p_i
\rangle \leq 1$ for all $i=1,\ldots,v$ and the vertices are $h_1,
\ldots, h_f$. Therefore, given a nonnegative $K^\ast$-map
$A:(\RR^f)^\ast \rightarrow (\RR^m)^\ast$ and a nonnegative $K$-map
$B:\RR^v \rightarrow \RR^m$ we can use them, as above, to define
approximations for $P^{\circ}$, $\Inn_{P^{\circ}}(B)$ and
$\Out_{P^{\circ}}(A)$, since facets and vertices of $P$ and $P^\circ$
exchange roles. By applying polarity again, we potentially get
new (different) outer and inner approximations of $P$ via
Proposition~\ref{prop:inclusions}. We now prove that in fact, we
recover the same approximations, and so in a sense, the approximations
are dual to each other.

As noted in Remark~\ref{rem:assume inn closed}, we are assuming that $\Inn_P(A)$ is closed.

\begin{theorem} 
\label{thm:polarity}
Let $P$ be a polytope, $A$ a nonnegative $K^\ast$-map and $B$ a nonnegative $K$-map as before. Then, 
$$\Out_{P^{\circ}}(A) = (\Inn_{P}(A))^{\circ},$$ and equivalently,
$$\Inn_{P^{\circ}}(B) = (\Out_{P}(B))^{\circ}.$$
\end{theorem}

\begin{proof} 
Note that the two equalities are equivalent as one can be obtained
from the other by polarity. Therefore, we will just prove the first
statement. For notational convenience, we identify $\RR^n$ and its
dual. By definition,
$$(\Inn_{P}(A))^{\circ} = \{ z \in \RR^n \,:\,  z^T x   \leq 1 \;\,\forall \, x \in \Inn_{P}(A) \}.$$
Consider then the problem
\[
\textup{max } \{  z^T x  \,:\, x \in \Inn_P(A) \} = \textup{max }  \{ z^Tx \,:\, {\mathbbm 1} - H^\ast(x) - A^{\ast}(y) \in \kin^f, \,\, y \in K\}.
\]
This equals the problem
$$\max_{x, y \in K} \, \min_{p \in (\kin^f)^*} 
  \{ z^T x + p^T ({\mathbbm 1} - H^\ast(x) - A^{\ast}(y)) \}.$$
  Strong duality  \cite{BoydVandenberghe} holds since the original problem 
  $\textup{max } \{  z^Tx \,:\, x \in \Inn_P(A) \}$ is a convex optimization problem and $\Inn_P(A)$ has an interior 
 as we will see in Proposition~\ref{prop:colcone error bound}. 
This allows us to switch the order of min and max, to obtain
\[
\min_{p \in (\kin^f)^*} \, \max_{x, y \in K} \{ {\mathbbm 1}^T p + x^T(z-H(p)) - A(p)^Ty \}.
\]
For this max to be bounded above, we need $z = H(p)$ since $x$
is unrestricted, and $A(p) \in K^\ast$ since $y \in K$. Therefore,
using $(\kin^f)^* = \kout^f$, we are left with
\[
\textup{min}_{p}  \{ {\mathbbm 1}^T p  \,:\, p \in \kout^f, \, \, z = H(p),\, \, A(p) \in K^{\ast} \}.
\]
Looking back at the definition of $(\Inn_{P}(A))^{\circ}$, we get
 $$(\Inn_{P}(A))^{\circ} = \{ H(p) \,:\, p \in \kout^f, \, \, {\mathbbm 1}^T p \leq 1, \, \, A(p) \in K^{\ast} \}$$ 
which is precisely $\Out_{P^{\circ}}(A)$.
\end{proof}

\subsection{Efficient representation}
There would be no point in defining approximations of $P$ if they
could not be described in a computationally efficient
manner. Remarkably, the orthogonal invariance of the $2$-norm
constraints in the definition of $\Inn_P(\cdot)$ and $\Out_P(\cdot)$
will allow us to compactly represent the approximations via second
order cones (SOC), with a problem size that depends only on the conic
rank $m$ (the affine dimension of $K$), and \emph{not} on the number
of vertices/facets of $P$. This feature is specific to our approximation and
is of key importance, since the dimensions of the codomain of $A^\ast$
and the domain of $B$ can be exponentially large.

We refer to any set defined by a $k \times k$ positive semidefinite
matrix $Q \succeq 0$ and a vector $a \in \RR^k$ of the form
\[
\{ z \in \RR^{k} \,:\, \|Q^{\frac{1}{2}} z \| \leq a^T z\}
\]
as a $k$-second order cone or $\textup{SOC}_k$. Here $Q^{\frac{1}{2}}$
refers to a $k \times k$ matrix that is a square root of $Q \succeq 0$, i.e., 
$Q = (Q^{\frac{1}{2}})^T Q^{\frac{1}{2}}$.  Suppose $M \in
\RR^{p \times k}$ is a matrix with $p \gg k$. Then $\| Mx \|^2 =
x^TM^TMx = x^TQx = \| Q^{\frac{1}{2}}x \|^2 $ where $Q = M^TM$.
 Therefore, the expression $\|Mx\|$ (which corresponds to a 2-norm
condition on $\RR^p$) is equivalent to a 2-norm condition on $\RR^k$
(and $k \ll p$). Notice that such a property is not true, for instance,
for any $\ell_q$ norm for $q \not = 2$.

We use this key orthogonal invariance property of $2$-norms to
represent our approximations efficiently via second-order
cones. Recall from the introduction that a convex set $C \subset
\RR^n$ has a $K$-lift, where $K \subset \RR^m$ is a closed convex
cone, if there is some affine subspace $L \subset \RR^m$ and a linear
map $\pi \,:\, \RR^m \rightarrow \RR^n$ such that $C = \pi(K \cap L)$.
A set of the form $K' = \{ z \in \RR^{k} \,:\, \|Q^{\frac{1}{2}} z
\|_2 \leq a_0 + a^Tz \}$ where the right-hand side is affine can be
gotten by slicing its homogenized second order cone
\[ 
\left\{ (z_0,z) \in \RR^{k+1} \,:\, \left\| \left( \begin{array}{cc} 0 &
  0 \\ 0 & Q^{\frac{1}{2}} \end{array} \right)\left( \begin{array}{c}
  z_0 \\ z \end{array} \right) \right\| \leq a_0z_0+a^T z \right\}
\]
with the affine hyperplane $\{ (z_0,z) \in \RR^{1+k} \,:\, z_0=1 \}$ and
projecting onto the $z$ variables. In other words, $K'$ has a
$\SOC_{k+1}$-lift.

\begin{theorem} 
\label{thm:small lifts}
Let $P \subseteq \RR^n$ be a polytope, $K \subseteq \RR^m$ a closed
convex cone and $A$ and $B$ nonnegative $K^\ast$ and $K$-maps
respectively. Then the convex sets $\Inn_{P}(A)$ and $\Out_P(B)$ have
$K \times \SOC_{(n+m+2)}$-lifts.
\end{theorem}  

\begin{proof} 
Set $M$ to be the $f \times (n+m+1)$ concatenated matrix $\left[
  H^\ast \,\, A^\ast \, \, \mathbbm{-1}\right]$. Then, using
Definition~\ref{def:inout}, the characterization in
Remark~\ref{rem:Kinalt} and $Q := M^TM$, we have
$$\begin{array}{lcl} 
\Inn_{P}(A) & = & \left\{ x \in \RR^n \,:\, \exists y \in K, \xi \in \RR \textup{ s.t. } 
\left\| M \left(  x^T \,\, y^T \, \, \xi \right)^T  \right\| \leq 1-\xi \right\}\\
& = & \left\{ x \in \RR^n \,:\, \exists y \in K, \xi \in \RR \textup{ s.t. } 
\| Q^{\frac{1}{2}} \left( x^T \,\, y^T \,\, \xi \right)^T \| \leq 1-\xi \right\}.
\end{array}$$

From the discussion above, it follows that $\Inn_P(A)$ is the
projection onto the $x$ variables of points $(z;x_0,x,y,\xi) \in K
\times \SOC_{n+m+2}$ such that $z=y$ and $x_0=1$ and so $\Inn_P(A)$
has a $K \times \SOC_{n+m+2}$-lift. 

The statement for $\Out_P(B)$ follows from Theorem~\ref{thm:polarity}
and the fact that if a convex set $C \subset \RR^n$ with the origin in
its interior has a $K$-lift, then its polar has a $K^\ast$-lift
\cite{GPT2012}. We also use the fact that the dual of a $\SOC_k$ is
another $\SOC_k$ cone.
\end{proof}

Note that the lifting dimension is \emph{independent} of $f$, the
number of facets of $P$, which could be very large compared to $n$. A
slightly modified argument can be used to improve this result by 1,
and show that a $K \times \SOC_{(n+m+1)}$-lift always exists.

\subsection{Approximation quality}
The final question we will address in this section is how good an
approximation $\Inn_P(A)$ and $\Out_P(B)$ are of the polytope $P$. 
We assume that $A^\ast(K)$ is closed which implies (by 
Remark~\ref{rem:A*(K) closed}) that both $\kin^f + A^\ast(K)$ and $\Inn_P(A)$ are closed.

One would like to prove that if we start with a good approximate
$K$-factorization of the slack matrix $S$ of $P$, we would get good
approximate lifts of $P$ from our definitions. This is indeed the case
as we will show. Recall that our approximations $\Inn_P(A)$ and 
$\Out_P(B)$ each depend on only one of the factors $A$ or $B$. 
For this reason, ideally, the goodness of these approximations 
should be quantified using only the relevant factor. 
Our next result presents a bound in this spirit.

\begin{definition} \label{def:mu}
For $x \in \RR^f$, let $\mu(x) := \textup{min} \{ t \geq 0 \,:\, x + t {\mathbbm 1} \in \kin^f + A^\ast (K) \}$.
\end{definition}

Note that $\mu(x)$ is well defined since $\kin^f + A^\ast (K)$ is closed  
and ${\mathbbm 1}$ lies in the interior of  $\kin^f + A^\ast (K)$. Also, $\mu(x) = 0$ for all $x \in \kin^f + A^\ast (K)$. 

\begin{proposition} \label{prop:colcone error bound}
For $\varepsilon = \textup{max}_i (\mu(S(e_i))$, $\frac{P}{1+\varepsilon} \subseteq \Inn_P(A)$.
\end{proposition}

\begin{proof}
It suffices to show that $\frac{p}{1+\varepsilon} \in \Inn_P(A)$ where $p = V(e_i)$ is a vertex of $P$.
By the definition of $\varepsilon$, we have that $\varepsilon {\mathbbm 1} + S(e_i) \in \kin^f + A^\ast(K)$ and hence, 
$\varepsilon {\mathbbm 1} + {\mathbbm 1} - H^\ast V(e_i) \in \kin^f + A^\ast(K)$. Therefore, 
${\mathbbm 1} - H^\ast(\frac{p}{1+\varepsilon}) \in \kin^f + A^\ast(K)$ and hence, 
$\frac{p}{1+\varepsilon} \in \Inn_P(A)$.
\end{proof}

From this proposition one sees that the approximation factor improves as $A^\ast(K)$ gets bigger.
While  geometrically appealing, this bound is not very convenient from a computational viewpoint. Therefore, we now 
write down a simple, but potentially weaker, result based on an alternative error measure $\xi(\cdot)$ defined as
follows.

\begin{definition} \label{def:xi}
For $x \in \RR^f$, define $ \xi(x) := \min \, \{ t \geq 0 \,: \, x + t {\mathbbm 1} \in \kin^f \}$.
\end{definition}
Again, $\xi(\cdot)$ is well-defined since $\kin^f$ is closed and ${\mathbbm 1}$ is in its interior.
Also, $\mu(x) \leq \xi(x)$ for all $x \in \RR^f$.  Setting $t = \|x\|$ in Remark~\ref{rem:Kinalt}, we see that 
for any $x \in \RR^f$, $x + \|x\| \cdot {\mathbbm 1} \in \kin^f$  which implies that $\xi(x) \leq \|x\|$. We also remark 
that $\xi(x)$ can be computed easily as follows.

\begin{remark}\label{rem:formula for xi}
For $0 \neq x \in \RR^f$, $f \geq 2$,   
\[
\xi(x) = \textup{max} \left\{0, \,\,\|x\| \cdot \alpha\left(\frac{{\mathbbm 1}^T x}{\sqrt{f}\|x\|}\right) \right\},
\]
where $\alpha(s) := (\sqrt{(f-1)(1-s^2)}-s)/\sqrt{f}$, and $\xi(0) = 0$.
\end{remark}

Using $\xi(\cdot)$ we will now provide a more convenient version of
Proposition~\ref{prop:colcone error bound} based only on the
factorization error.

\begin{lemma} \label{lem:cone eigenvalue}
Let $A:(\RR^f)^\ast\rightarrow (\RR^m)^\ast$ be a nonnegative
$K^\ast$-map and $B: \RR^v \rightarrow \RR^m$ be a nonnegative
$K$-map. Let $\Delta=S - A^\ast \circ
B$. Then,
$\frac{1}{1+\varepsilon} P \subseteq \Inn_{P}(A)$, for $\varepsilon = 
\textup{max}_i \, \xi(\Delta(e_i))$.
\end{lemma}

\begin{proof}
Note that $\mu(S(e_i)) = \min \{ t \geq 0 \,:\, \exists \,\, u \in A^\ast(K) \textup{ s.t. } S(e_i)-u + t {\mathbbm 1} \in \kin^f \}$. 
Since $B(e_i) \in K$, we have that $A^\ast(B(e_i)) \in A^\ast(K)$. Therefore, 
\[
\mu(S(e_i))  \leq \min\{t \geq 0 \,: \, \underbrace{S(e_i) -A^\ast(B(e_i))}_{\Delta(e_i)} + t {\mathbbm 1} \in \kin\} = \xi(\Delta(e_i)).
\]
\end{proof}

%We remark that the function $\xi(\cdot)$ in fact, has a simple formula.
%
%\begin{lemma} \label{lem:formula for xi}
%For $0 \neq x \in \RR^f$, $f \geq 2$,   
%\[
%\xi(x) := \textup{max} \left\{ 0, \|x\| \cdot \alpha\left(\frac{{\mathbbm 1}^T x}{\sqrt{f}\|x\|}\right) \right\},
%\]
%where $\alpha(s) := (\sqrt{(f-1)(1-s^2)}-s)/\sqrt{f}$. Also, $\xi(x) \leq \|x\|$.
%\end{lemma}
%
%\begin{proof} 
%Notice that $\xi(x)$ is well-defined since $-1 \leq \frac{{\mathbbm
%    1}^T x}{\sqrt{f}\|x\|} \leq 1$ by Cauchy-Schwarz.  For $t =
%\xi(x)$, one can verify that $\sqrt{f-1} \|x + t {\mathbbm 1} \| =
%   {\mathbbm 1}^T (x + t{\mathbbm 1})$ (and it is the largest such
%   value), and thus $x + t {\mathbbm 1}$ is on the boundary of
%   $\kin$. Since ${\mathbbm 1} \in \kin^f$, increasing the value of $t$
%   will keep $x + t {\mathbbm 1}$ in $\kin^f$. For the second claim, an
%   easy calculation shows that $\alpha(s) \leq 1$ for all $s \in
%   [-1,1]$ (since it is concave, and achieves this maximum at
%   $s=-\frac{1}{\sqrt{f}}$), so the result follows.
%\end{proof}

We immediately get our main result establishing the connection between
the quality of the factorization and the quality of the
approximations. For simplicity, we state it using the simplified upper
bound $\xi(x) \leq \|x\|$.

\begin{proposition}\label{prop:errorbounds}
Let $A:(\RR^f)^\ast\rightarrow (\RR^m)^\ast$ be a nonnegative
$K^\ast$-map and $B: \RR^v \rightarrow \RR^m$ be a nonnegative
$K$-map. Let $\Delta :=S - A^\ast \circ B$ be the factorization error. Then,
\begin{enumerate}
\item $\frac{1}{1+\varepsilon} P \subseteq \Inn_{P}(A)$, for $\varepsilon = \| \Delta \|_{1,2}$;
\item $\Out_{P}(B) \subseteq (1+\varepsilon) P$, for $\varepsilon  = \| \Delta \|_{2,\infty}$,
\end{enumerate}
where $\| \Delta \|_{1,2} = \max_i \| \Delta(e_i) \|$ is the induced
$\ell_1,\ell_2$ norm and $\| \Delta \|_{2,\infty} = \max_i \|
\Delta^\ast(e_i^\ast) \|$ is the induced $\ell_2,\ell_\infty$ norm of
the factorization error.
\end{proposition}

\begin{proof}
By Theorem~\ref{thm:polarity}, the two statements are equivalent.
The proof now follows from Lemma~\ref{lem:cone eigenvalue}.
\end{proof}

This means that if we start with $A$ and $B$ forming a
$K$-factorization of a nonnegative $S'$ which is close to the true
slack matrix $S$, we get a $(1+\|S-S'\|_{1,2})$-approximate inner
approximation of $P$, as well as a
$(1+\|S-S'\|_{2,\infty})$-approximate outer approximation of
$P$. Thus, good approximate factorizations of $S$ do indeed lead to
good approximate lifts of $P$.

\begin{example} \label{ex:inclusions}
Consider the square given by the inequalities $1 \pm x \geq 0$ and $1
\pm y \geq 0$ with vertices $(\pm 1, \pm 1)$.  By suitably ordering
facets and vertices, this square has slack matrix
$$S=\left(
\begin{array}{cccc}
2 & 2 & 0 & 0\\
0 & 2 & 2 & 0\\
0 & 0 & 2 & 2\\
2 & 0 & 0 & 2
\end{array}
\right).$$ Let $A:(\RR^4)^\ast\rightarrow (\RR^2)^\ast$ and
$B\,:\,\RR^4 \rightarrow \RR^2$ be the nonnegative maps given by the
matrices
$$A=\left(
\begin{array}{cccc}
4/3 & 4/3 & 0 & 0\\
0 & 0 & 4/3 & 4/3
\end{array}
\right) \,\, \textup{ and } \,\, B=\left(
\begin{array}{cccc}
1 & 1 & 1 & 0\\
1 & 0 & 1 & 1
\end{array}
\right).$$
Then $A$ and $B$ are nonnegative $\RR^2_+ = (\RR^2_+)^\ast$ maps and
$$A^T B = \left(
\begin{array}{cccc}
4/3 & 4/3 & 4/3 & 0\\
4/3 & 4/3 & 4/3 & 0\\
4/3 & 0 & 4/3 & 4/3\\
4/3 & 0 & 4/3 & 4/3
\end{array}
\right).$$ 
It is easy to check that $\|S-A^T B\|_{1,2}=\frac{2}{3} \sqrt{10}$ while
$\|S-A^T B\|_{2,\infty}=2\sqrt{\frac{2}{3}}$. So by
Proposition~\ref{prop:errorbounds} this implies
that $$\frac{1}{1+\frac{2}{3}\sqrt{10}}P \subseteq \Inn_P(A) \subseteq P
\subseteq \Out_P(B) \subseteq (1 + 2\sqrt{\textstyle\frac{2}{3}})P.$$

If we use instead Lemma~\ref{lem:cone eigenvalue}, we can get the slightly better quality bounds
 $$\frac{1}{\frac{4}{3}+\sqrt{3}}P \subseteq \Inn_P(A) \subseteq P \subseteq \Out_P(B) \subseteq (1 + \sqrt{2})P.$$
 
Finally, if we use directly Proposition~\ref{prop:colcone error
  bound}, it is possible in this case to compute explicitly the true
bounds
 $$\frac{1}{\sqrt{3}}P \subseteq \Inn_P(A) \subseteq P \subseteq \Out_P(B) \subseteq \sqrt{3}P.$$
In Figure \ref{Fig:epsapprox} we can see the relative quality of all these bounds.

\begin{figure}
\centering
\includegraphics[scale=0.4]{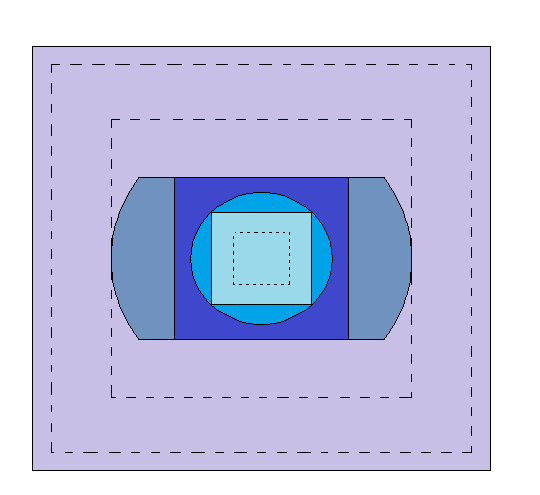}
\caption{$\Inn_{P}(A)$, $P$ and $\Out_P(B)$, as well as all the guaranteed approximations.}
\label{Fig:epsapprox}
\end{figure}

\end{example}

%Another natural definition of an approximate factorization of $S$
%might be one where $S=A^\ast \circ B$ but the columns of $A$ are only
%required to be close to $K$ and not in $K$ and similarly, the columns
%of $B$ are required to be only close to $K^\ast$ and not in
%$K^\ast$. For instance, an approximation to nonnegative factorization
%of $S$ where $S=A^TB$ but $A$ and $B$ could have some negative
%entries. As before, one can define $\Inn_P(B)$ and
%$\Out_P(A)$. However, we would not be able to guarantee that
%$\Inn_P(B)$ is inside $P$. In fact, $\Inn_P(B)$ could be unbounded.

%%%%%%%%%%%%%%%%

\section{Special Cases} 
\label{sec:cases}

In this section we relate our inner approximation to \emph{Dikin's
  ellipsoid}, a well known inner approximation to a polytope that
arises in the context of interior point methods.  We also compute
$\Inn_P(A)$ and $\Out_P(B)$ from a nonnegative factorization of the
closest rank one approximation of the slack matrix of $P$. As we will
see, these are the two simplest possible approximations of a polytope,
and correspond to specific choices for the factors $A$ and $B$.

\subsection{Dikin's ellipsoid}
Recall that $\Inn_{P}(0)$ is an ellipsoidal  inner
approximation of $P$ that is intrinsic to $P$ and does not depend on any approximate
factorization of the slack matrix of $P$ through any cone. A commonly used inner approximation of a polytope is 
Dikin's ellipsoid defined as follows.
Let $P = \{ x \in \RR^n \,:\, H^Tx \leq {\mathbbm 1} \}$ be a polytope as before and let $h_1, \ldots, h_f$ be the columns of $H$. Define
\[
\Gamma_P(x) := \sum_{i=1}^f \frac{h_i h_i^T}{(1-\langle h_i,x \rangle)^2},
\]
which is the Hessian of the standard logarithmic barrier function
$\psi(x):=-\sum_{i=1}^f \log (1-\langle h_i,x \rangle)$ associated to
$P$. If $x_0$ is in the interior of $P$, then the Dikin ellipsoid
centered at $x_0$ and of radius $r$, is the set
\[
D_{x_0}^r :=\{x \in \RR^n \, : \, (x-x_0)^T \Gamma_P(x_0)(x-x_0) \leq r^2 \}.
\]
It can be checked that $D_{x_0}^1 \subseteq P$ (see Theorem 2.1.1 in
\cite{NN}). Further, since Dikin's ellipsoid is invariant under
translations, we may assume that $x_0 = 0$,  and then \
\[
D_0^1=
\left\{x \in \RR^n \, : \, x^T \left(\sum_{i=1}^f {h_i h_i^T}\right) x \leq 1 \right\} =
\left\{ x \in \RR^n \,: \,\ \|H^Tx\| \leq 1 \right\}.
\] 
Recall that 
\begin{align*}
\Inn_P(0) &= \left\{ x \in \RR^n \, : \, {\mathbbm 1} - H^Tx \in \kin^f \right\}
= \left\{ x \in \RR^n \, : \, \exists t \in
\RR \mbox{ s.t. } \| (t-1) {\mathbbm 1} + H^T x\| \leq t \right\},
\end{align*}
where we used the characterization of $\kin^f$ given in
Remark~\ref{rem:Kinalt}. Choosing $t=1$, we see that $D_0^1 \subseteq
\Inn_P(0) \subseteq P$. This inclusion is implicit
in the work of Sturm and Zhang \cite{SturmZhang}. \marginpar{check}

If the origin is also the analytic center of $P$ (i.e., the minimizer
of $\psi(x)$), then we have that $P \subseteq \sqrt{f(f-1)}D_{0}^1$
where $f$ is the number of inequalities in the description of $P$; see
\cite{Sonnevend} and \cite[Section 8.5.3]{BoydVandenberghe}. Also, in this
situation, the first order optimality condition on $\psi(x)$ gives
that $\sum_{i=1}^f h_i = 0$, and as a consequence, $f= \sum_{i=1}^f
(1-\langle h_i, p \rangle)$ for all vertices $p$ of $P$. In other
words, every column of the slack matrix sums to $f$.
%and as a consequence, $\sum_{i=1}^f (1 - \langle h_i, x \rangle) = f$
%for all $x \in P$. Taking $x$ to be a vertex of $P$, we then get that
%the sum of the entries in every column of the slack matrix $S$ is $f$.
This implies an analogous (slightly stronger) containment result for
$\Inn_P(0)$.

\begin{corollary}
If the origin is the analytic center of the polytope $P$, then 
\[
\Inn_P(0) \subseteq P \subseteq (f-1) \, \Inn_P(0),
\]
furthermore, if $P$ is centrally symmetric
\[
\Inn_P(0) \subseteq P \subseteq \sqrt{f-1} \, \Inn_P(0).
\]
\end{corollary}

\begin{proof} 
This follows from Lemma~\ref{lem:cone eigenvalue} and Remark~\ref{rem:formula for xi}, 
 by using the fact
that if $S$ is the slack matrix of~$P$ then $\|S(e_i) \| \leq
{\mathbbm 1}^TS(e_i) = f$. In this case, for $w = S(e_i)$, we have
$\frac{\mathbbm{1}^T w}{\sqrt{f} \|w\|} \geq \frac{1}{\sqrt{f}}$ and
thus, since $\alpha(s)$ is decreasing for $s \geq
-\frac{1}{\sqrt{f}}$, we get $\xi(w) = \|w\| \alpha(\frac{\mathbbm{1}^T
  w}{\sqrt{f} \|w\|}) \leq \|w\| \alpha(\frac{1}{\sqrt{f}}) \leq f
\cdot \frac{f-2}{f} = f-2 $, from where the first result follows.

For the second result, note that if $P$ is centrally symmetric, for
every facet inequality $1-\langle h_j, x \rangle \geq 0$ we have also
the facet inequality $1+ \langle h_j, x \rangle \geq 0$, which implies
\[
\|S(e_i)\|=\sqrt{\frac{1}{2} \sum_{j=1}^f \left[(1- \langle h_j, p_i \rangle)^2 + (1+ \langle h_j, p_i \rangle)^2 \right]}=\sqrt{f+\|H^T p_i\|^2}.
\]
Using this fact together with ${\mathbbm 1}^TS(e_i) = f$, we get that
$\xi(S(e_i))=\sqrt{\frac{f-1}{f}}\|H^T p_i\|-1$. Since $\mathbbm{1}-H^T p_i$ and $\mathbbm{1}+H^T p_i$ are
both nonnegative, all entries in $H^T p_i$ are smaller than one in absolute
value. Hence, $\|H^T p_i\| \leq \sqrt{f}$, concluding the proof.
\end{proof}

\begin{remark}
It was shown by Sonnevend \cite{Sonnevend} that when $x_0$ is the analytic center of $P$, 
$ \sqrt{{f}/(f-1)} D_{x_0}^1  \subseteq P$. By affine invariance of the result, we may assume that $x_0=0$, and in combination with the previously mentioned result from \cite{Sonnevend}, we get 
$$\sqrt{{f}/(f-1)} D_{0}^1  \subseteq P \subseteq \sqrt{f(f-1)} D_0^1.$$
This implies that the {\em Sonnevend ellipsoid},  $\sqrt{{f}/(f-1)} D_{0}^1$,  a slightly dialated version of the Dikin ellipsoid $D_0^1$, is also contained in $P$. 
Since $D_0^1 = \{ x \in \RR^n \,:\, \|H^Tx\| \leq {\mathbbm 1} \}$, we get that $\sqrt{{f}/(f-1)} D_{0}^1 = \{ x \in \RR^n \,:\, \| H^Tx\| \leq \sqrt{{f}/(f-1)} \}$. 
On the other hand, 
$$\begin{array}{lcl}
\Inn_P(0) & = & \{ x \in \RR^n \,:\, {\mathbbm 1} - H^Tx \in \kin^f \}\\
&  = & \{ x \in \RR^n \,:\, \sqrt{f-1} \|{\mathbbm 1} - H^Tx \| \leq {\mathbbm 1}^T({\mathbbm 1} - H^Tx) \}\\
& = & \{ x \in \RR^n \,:\, (f-1) \|H^Tx\|^2 \leq f - 2(\sum h_i)^Tx + ((\sum h_i)^Tx)^2 \}.
\end{array}.$$ 
When the origin is the analytic center of $P$, $\sum h_i = 0$, and $\Inn_P(0) =  \sqrt{{f}/(f-1)} D_{0}^1$, which makes the containment $\sqrt{{f}/{f-1}} D_{0}^1  \subseteq P$ at most as strong as $\Inn_P(0) \subseteq P$. The containment  $\Inn_P(0) \subseteq P$ holds whenever the origin is in $P$ while the Sonnevend containments require the origin to be the analytic center of $P$.
\end{remark}

\subsection{Singular value decomposition}
Let $P$ be a polytope as before and suppose
$S=U \Sigma V^T$ is a singular value decomposition of the slack matrix $S$ of $P$ 
with $U$ and $V$ orthogonal matrices and $\Sigma$ a diagonal matrix with the singular values of $S$ on the diagonal. 
By the Perron-Frobenius theorem, the leading singular vectors of a nonnegative
matrix can be chosen to be nonnegative. Therefore, if $\sigma$ is the
largest singular value of $S$, and $u$ is the first column of $U$ and 
$v$ is the first column of $V$ then $A=\sqrt{\sigma}u^T$ and $B=\sqrt{\sigma}v^T$ are two
nonnegative matrices. By the Eckart-Young
theorem, the matrix $A^TB = \sigma u v^T$ is the closest rank one matrix (in Frobenius norm) to 
$S$. We can also view $A^TB$ as an approximate $\RR_+$-factorization of $S$ and thus look at the 
approximations $\Inn_P(A) =:\Inn_{P}^{\textrm{sing}}$ and $\Out_P(B) =: \Out_{P}^{\textrm{sing}}$ and hope  that in some cases they may offer good compact approximations to $P$. Naturally,
these are not as practical as $\Inn_P(0)$ and $\Out_P(0)$ since to
compute them we need to have access to a complete slack matrix of $P$ and its  leading singular vectors.
We illustrate these approximations on an example.

\begin{example}
Consider the quadrilateral with vertices $(1,0), (0,2), (-1,0)$ and $(0,-1/2)$. This polygon has slack matrix
$$S=\left[
\begin{array}{cccc}
0 & 5 & 2 & 0 \\
0 & 0 & 2 & 5/4 \\
2 & 0 & 0 & 5/4 \\
2 & 5 & 0 & 0
\end{array}
\right],$$ 
and by computing a singular value decomposition and
proceeding as outlined above, we obtain the $1 \times 4$ matrices 
$$A=[1.9130,0.1621,0.1621,1.9130] \textup{ and } 
B=[0.5630,2.5951,0.5630,0.0550],$$ verifying
$$S'=A^TB=\left[
\begin{array}{cccc}
    1.0770  &  4.9644 &   1.0770 &   0.1051\\
    0.0912  &  0.4206  &  0.0912 &  0.0089\\
    0.0912  &  0.4206 &   0.0912  &  0.0089\\
    1.0770  &  4.9644 &  1.0770  &  0.1051
    \end{array}
\right].$$ Computing $\Inn_{P}^{\textrm{sing}}$ and
$\Out_{P}^{\textrm{sing}}$ we get the inclusions illustrated in
Figure~\ref{Fig:quadincl}. Notice how these rank one approximations
from leading singular vectors use the extra information to improve on
the trivial approximations $\Inn_P(0)$ and $\Out_P(0)$. Notice also
that since $$\Out_P^{\textrm{sing}} = \Out_P(B) = \{ V(z) \,:\, z \in
\RR^4, \, \|z\| \leq {\mathbbm 1}^Tz \leq 1, \, Bz \geq 0 \}, $$ and
$Bz \geq 0$ is a single linear inequality, $\Out_P^{\textrm{sing}}$ is
obtained from $\Out_P(0)$ by imposing a new linear inequality. By
polarity, $\Inn_{P}^{\textrm{sing}}$ is the convex hull of $\Inn_P(0)$
with a new point.

\begin{figure}
\centering
\includegraphics[scale=0.35]{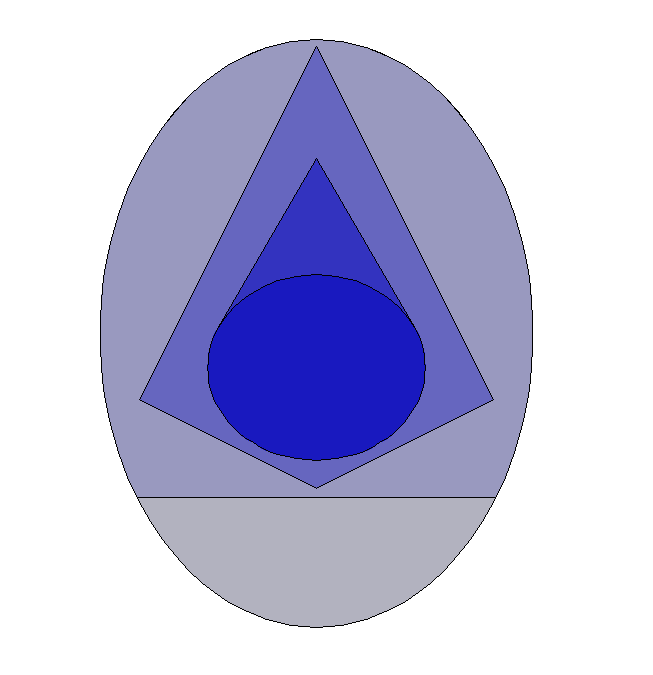}
\caption{$\Inn_{P}(0) \subseteq \Inn_{P}^{\textrm{sing}} \subseteq P \subseteq \Out_{P}^{\textrm{sing}} \subseteq \Out_P(0)$.}
\label{Fig:quadincl}
\end{figure}

\end{example}

\section{Nested polyhedra and Generalized Slack Matrices}
\label{sec:twomatrices}

In this section, we consider approximate factorizations of generalized
slack matrices and what they imply in terms of approximations to a
pair of nested polyhedra. The results here can be seen as
generalizations of results in Section~2.

Let $P \subset \RR^n$ be a polytope with vertices $p_1, \ldots, p_v$ and the origin in its interior,  
and $Q \subset \RR^n$ be a polyhedron with facet inequalities $\langle h_j, x \rangle 
\leq 1$, for $j=1,\ldots,f$, such that $P \subseteq Q$. The {\em slack
  matrix} of the pair $P,Q$ is the $f \times v$ matrix $S_{P,Q}$ whose
$(i,j)$-entry is $1 - \langle h_i, p_j \rangle$.  In the language of operators from Section~2,
$S_{P,Q}$ can be thought of as an operator from $\RR^v
\rightarrow \RR^f$ defined as $S_{P,Q}(x) = (\mathbbm{1}_{f \times v} -
H_Q^\ast \circ V_{P})(x)$ where $V_P$ is the vertex operator of $P$
and $H_Q$ is the facet operator of $Q$. Every nonnegative
matrix can be interpreted as such a generalized slack matrix after a
suitable rescaling of its rows and columns, possibly with some
extra rows and columns representing redundant points in $P$ and
redundant inequalities for $Q$.

\subsection{Generalized slack matrices and lifts}
Yannakakis' theorem about $\RR^m_+$-lifts of polytopes can be extended
to show that $S_{P,Q}$ has a $K$-factorization (where $K \subset
\RR^m$ is a closed convex cone), if and only if there exists a convex
set $C$ with a $K$-lift such that $P \subseteq C \subseteq Q$. (For a
proof in the case of $K = \RR^m_+$, see \cite[Theorem
  1]{BraunFioriniPokuttaSteurer}, and for $K = \PSD^m$ see
\cite{GRT2013}.  Related formulations in the polyhedral situation also
appear in \cite{GillisGlineur, Pashkovich}.)  For an arbitrary cone $K$,  
there is a requirement that the $K$-lift of $C$ contains an interior point of $K$ \cite[Theorem 1]{GPT2012},
but this is not needed for nonnegative orthants or psd cones, and we 
ignore this subtlety here.

Two natural convex sets with $K$-lifts that are nested between $P$ and
$Q$ can be obtained as follows.  For a nonnegative $K^\ast$-map $A
\,:\, (\RR^f )^\ast\rightarrow (\RR^m)^\ast$ and a nonnegative $K$-map
$B: \RR^v \rightarrow \RR^m$, define the sets:
\[
\begin{aligned}
C_A &:= \left\{ x \in \RR^n \, : \, \exists \,y \in K \textup{ s.t. }
{\mathbbm{1}} - H_Q^\ast(x) - A^{\ast}(y) = 0 \right\},\\
C_B &:= \{ V_P(z) \,:\, {\mathbbm 1}^T z = 1, \,\, B(z) \in K \}.
\end{aligned}
\]
The sets $C_A$ and $C_B$ have $K$-lifts since they are obtained by
imposing affine conditions on the cone $K$.  From the definitions
and Remark~\ref{rem:alt Out}, it immediately follows that these containments
hold:
\[
P \subseteq \Out_P(B) \subseteq C_B
\quad \textup{ and } \quad 
C_A \subseteq \Inn_Q(A) \subseteq Q .
\]

As discussed in the introduction, the set $C_A$ could potentially be
empty for an arbitrary choice of $A$. On the other hand, since $P
\subseteq C_B$, $C_B$ is never empty. Thus in general, it is not true
that $C_B$ is contained in $C_A$. However, in the presence of an exact
factorization we get the following chain of containments.

%In particular, $C := C_A = C_B$ is not empty and provides a convex set with a $K$-lift that 
%is sandwiched between $P$ and $Q$.

\begin{proposition} \label{prop:containments}
When $S_{P,Q} = A^\ast \circ B$, we get $P \subseteq \Out_P(B)
\subseteq C_B \subseteq C_A \subseteq \Inn_Q(A) \subseteq Q$.
\end{proposition}

\begin{proof}
We only need to show that $C_B \subseteq C_A$. Let $V_P(z) \in C_B$, with ${\mathbbm 1}_v^Tz = 1$ and $B(z) \in K$. Then 
choosing $y = B(z)$, we have that ${\mathbbm 1}_f - H_Q^\ast(V_P(z)) - A^\ast(B(z)) 
={\mathbbm 1}_f {\mathbbm 1_v}^T z- H_Q^\ast(V_P(z)) - A^\ast(B(z)) = S_{P,Q}(z) - A^\ast(B(z))
= 0$ since $S_{P,Q} = A^\ast \circ B$.
Therefore, $C_B \subseteq C_A$ proving the result.
 \end{proof}

\begin{example}\label{ex:nested}
To illustrate these inclusions, consider the quadrilateral $P$ with
vertices $(1/2,0), (0,1), (-1/2,0), (0,-1)$, and the quadrilateral $Q$ 
defined by the inequalities $1-x \geq 0$, $1-x/2-y/2 \geq 0$, $1+x \geq 0$ and $1+x/2+y/2 \geq 0$. We have $P
\subseteq Q$, and 
$$S_{P,Q}=
\frac{1}{4}\left(
\begin{array}{cccc}
2 & 4 & 6 & 4 \\
3 & 2 & 5 & 6 \\
6 & 4 & 2 & 4 \\
5 & 6 & 3 & 2 
\end{array}
\right).$$
For $K=\RR^3_+$ we can find an exact factorization for this matrix, such as the one given by
$$ S_{P,Q} = A^T B = \left( \begin{array}{ccc} 0 & 2 & 1 \\ 0 & 1 & \frac{3}{2} \\ 2 & 0 & 1\\ 2 & 1 & \frac{1}{2} \end{array} \right) 
\left( \begin{array}{cccc} \frac{1}{2} &\frac{1}{2} &0 & 0 \\ 0 & \frac{1}{2} & \frac{1}{2} & 0 \\ \frac{1}{2} & 0 & \frac{1}{2} & 1 
\end{array} \right).$$
In this example, $$C_A = C_B = \left \{ (x_1, x_2) \,:\, 1 - x_2 \geq 0, 1 + 2x_1 + x_2 \geq 0, 1 - 2x_1 + x_2 \geq 0 \right \}.$$
By computing $\Inn_Q(A)$ and $\Out_P(B)$ we can see as in Figure~\ref{fig:nestednew} that 
$$ P \subsetneq \Out_P(B) \subsetneq C_B = C_A \subsetneq \Inn_Q(A) \subsetneq Q.$$

\begin{figure}[t]
\centering
\includegraphics[trim= 0mm 20mm 0mm 15mm, clip, height=6cm]{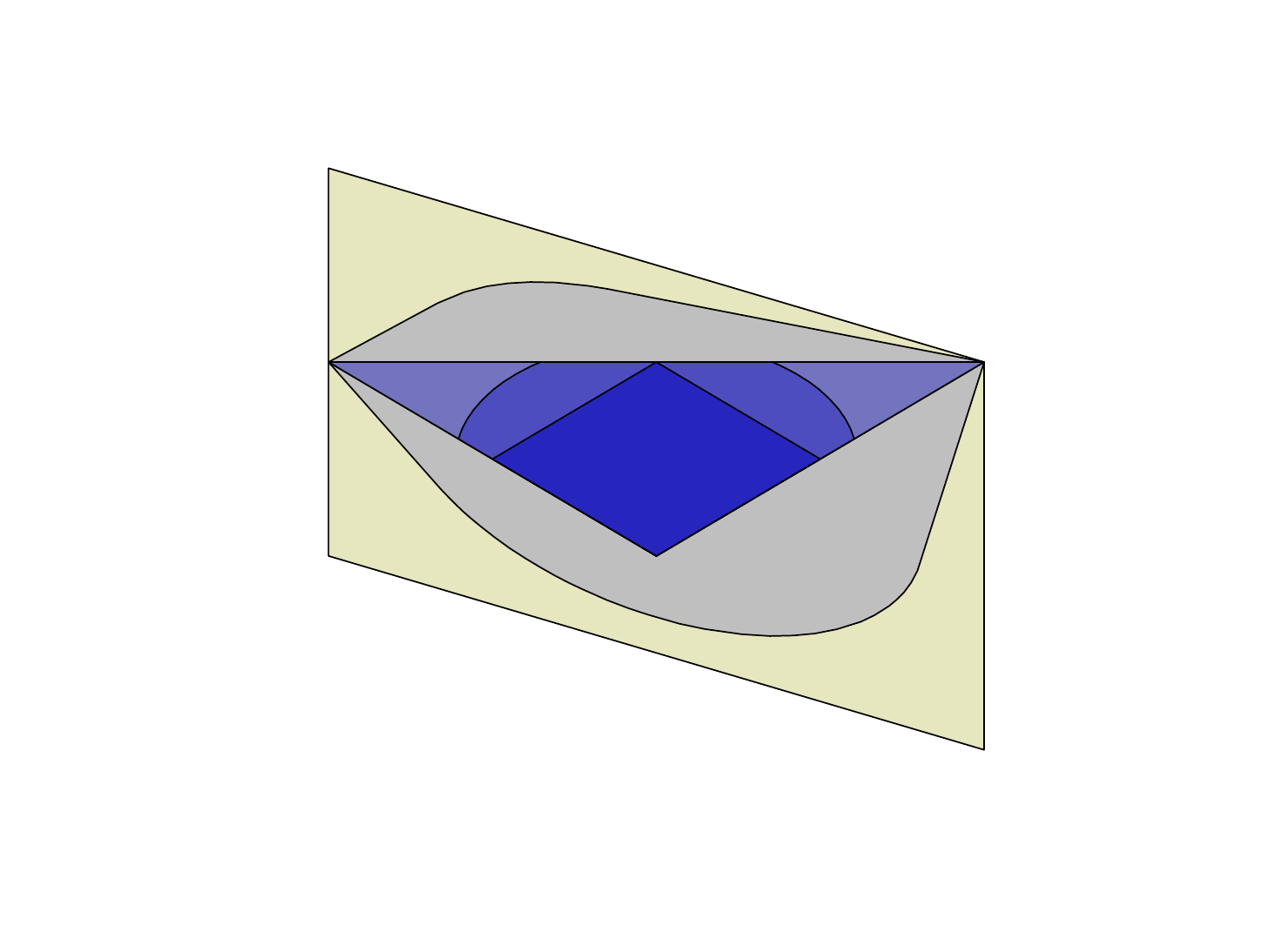}
\caption{$P \subsetneq \Out_P(B) \subsetneq C_B = C_A \subsetneq \Inn_Q(A) \subsetneq Q$.}
\label{fig:nestednew}
\end{figure}

Note that if instead we pick the factorization
$$ S_{P,Q} = A^T B = 
\left( \begin{array}{cccc} 
0 & 2 & 1 & 1\\ 
0 & 1 & \frac{3}{2} & 0\\ 
2 & 0 & 1 & 1\\ 
2 & 1 & \frac{1}{2} & 2
\end{array} \right) 
\left( \begin{array}{cccc} 
\frac{1}{2} &\frac{1}{2} &0 & 0 \\ 
0 & \frac{1}{2} & \frac{1}{2} & 0 \\ 
\frac{1}{2} & 0 & \frac{1}{2} & 1 \\
0 & 0 & 0 & 0
\end{array} \right),$$
then $C_A \subsetneq C_B$ and we get the inclusions in Figure~\ref{fig:nestednew2}.

\begin{figure}[t]
\centering
\includegraphics[trim= 0mm 20mm 0mm 15mm, clip, height=6cm]{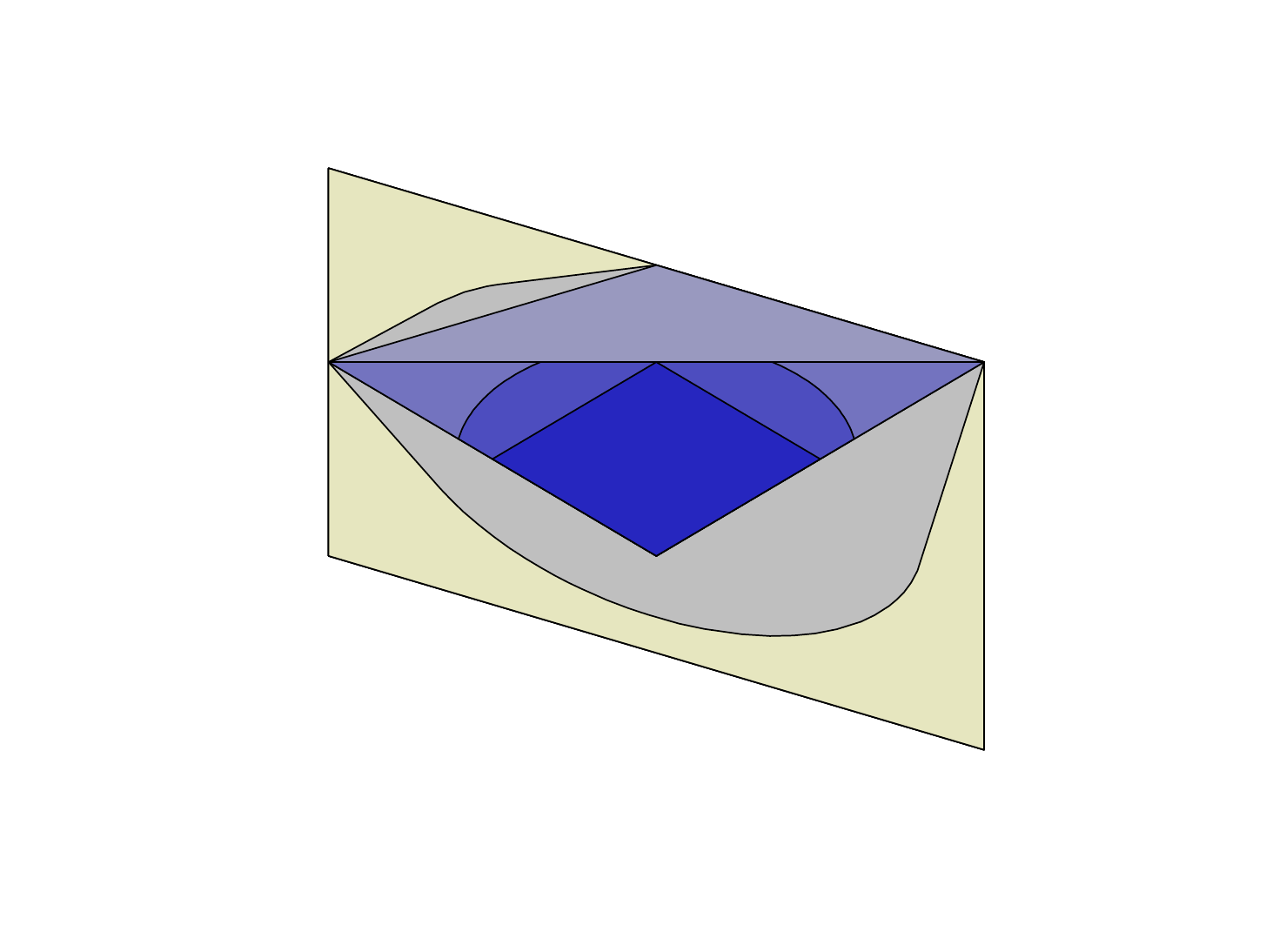}
\caption{$P \subsetneq \Out_P(B) \subsetneq C_B \subsetneq C_A \subsetneq \Inn_Q(A) \subsetneq Q$.}
\label{fig:nestednew2}
\end{figure}
\end{example}

\subsection{Approximate Factorizations of $S_{P,Q}$}
We now consider approximate factorizations of the generalized slack
matrix $S_{P,Q}$ and provide three results about this
situation. First, we generalize Proposition \ref{prop:errorbounds} to
show that the quality of the approximations to $P$ and $Q$ is directly
linked to the factorization error.  Next we show how approximations to
a pair of nested polytopes yield approximate $K$-factorizations of
$S_{P,Q}$. We close with an observation on the recent
inapproximability results in \cite{BraunFioriniPokuttaSteurer} in the
context of this paper.

When we only have an approximate factorization of $S_{P,Q}$, we obtain
a strict generalization of Proposition~\ref{prop:errorbounds}, via the
same proof.

\begin{proposition}\label{prop:errorbounds2}
Let $A:
(\RR^f)^\ast \rightarrow (\RR^m)^\ast$ a nonnegative $K^\ast$-map and 
$B:\RR^v\rightarrow \RR^m$ be a nonnegative $K$-map. Let
$\Delta :=S_{P,Q} - A^\ast \circ B$ be the factorization error.
%operator that records the difference between $S_{P,Q}$ and $A^\ast
%\circ B$, and $E^{\ast} =({\mathbbm 1}-V_P^\ast \circ H_Q) - B^\ast
%\circ A$. 
Then,
\begin{enumerate}
\item $\frac{1}{1+\varepsilon} P \subseteq \Inn_{Q}(A) \subseteq Q$,
  for $\varepsilon = \|\Delta\|_{1,2}$.
\item $P \subseteq \Out_{P}(B) \subseteq (1+\varepsilon) Q$, for
  $\varepsilon =\|\Delta\|_{2,\infty}$.
\end{enumerate}
\end{proposition}

So far we discussed how an approximate factorization of the slack
matrix can be used to yield approximations of a polytope or a pair of
nested polyhedra. It is also possible to go in the opposite direction
in the sense that approximations to a pair of nested polyhedra $P
\subseteq Q$ yield approximate factorizations of the slack matrix
$S_{P,Q}$ as we now show.

\begin{proposition} \label{prop:approximation to factorization}
Let $P \subseteq Q$ be a pair of polyhedra as before and $S_{P,Q}$ be
its slack matrix.  Suppose there exists a convex set $C$ with a
$K$-lift such that $\alpha P \subseteq C \subseteq \beta Q$ for some
$0 < \alpha \leq 1$ and $\beta \geq 1$. Then there exists a
$K$-factorizable nonnegative matrix $S'$ such that $|
(S'-S_{P,Q})_{ij} |  \leq \frac{\beta}{\alpha} - 1$.
\end{proposition}

\begin{proof}
Assume that the facet inequalities of $Q$ are of the form $\langle h_j, x
\rangle \leq 1$.  Since $C$ has a $K$-lift, it follows from
\cite{GPT2012} that we can assign an element $b_x \in K$ to each point
$x$ in $C$, and an element $a_y \in K^\ast$ to every valid inequality
$y_0-\left<y,x\right> \geq 0$ for $C$, such that $\left<a_y,b_x\right>
= y_0-\left<y,x\right>$.

If $\alpha = \beta = 1$, then $P \subseteq C \subseteq Q$, and since
$C$ has a $K$-lift, $S_{P,Q}$ has a $K$-factorization and $S' =
S_{P,Q}$ gives the result. So we may assume that $\frac{\beta}{\alpha}
>1$ and define $\eta > 0$ such that $(1 + \eta) \alpha = \beta$.  If
$1-\left<h,x\right> \geq 0$ defines a facet of $Q$, then
$1+\eta-\frac{1}{\alpha}\left<h,x\right> \geq 0$ is a valid inequality for 
$(1+\eta) \alpha Q = \beta Q$, and is therefore a valid inequality for
$C$. Hence, we can pick $a_1, \ldots, a_f$ in $K^*$, one for each
such inequality  as mentioned above. 
Similarly, if $v$ is a
vertex of $P$, then $\alpha v$ belongs to $C$ and we can pick
$b_1, \ldots, b_v$ in $K$, one for each $\alpha v$. Then,
$$\left<a_i,b_j\right> = 1+\eta-\frac{1}{\alpha}\left<h_i,\alpha
v_j\right> = \eta + (S_{P,Q})_{ij}.$$ The matrix $S'$ defined by
$S'_{ij} :=\left<a_i,b_j\right>$ yields the result.
\end{proof}

Note that Proposition~\ref{prop:approximation to factorization} is not
a true converse of Proposition
\ref{prop:errorbounds2}. Proposition~\ref{prop:errorbounds2} says that
approximate $K$-factorizations of the slack matrix give approximate
$K\times \textup{SOC}$-lifts of the polytope or pair of polytopes,
while Proposition~\ref{prop:approximation to factorization} says that
an approximation of a polytope with a $K$-lift gives an approximate
$K$-factorization of its slack matrix. We have not ruled out the
existence of a polytope with no good $K$-liftable approximations but
whose slack matrix has a good approximate $K$-factorization.

A recent result on the inapproximability of a polytope by
polytopes with small lifts, come from the max clique problem, as seen
in \cite{BraunFioriniPokuttaSteurer} (and strengthened in
\cite{BravermanMoitra}).  In \cite{BraunFioriniPokuttaSteurer}, the
authors prove that for $P(n)=\textup{COR}(n)= \textup{conv}\{bb^T\ | \ b \in
\{0,1\}^n\}$ and
$$Q(n)=\{x \in \RR^{n \times n} \ | \ \left<2 \textup{diag}(a)-aa^T,x
\right> \leq 1, a \in \{0,1\}^n\},$$ and for any $\rho > 1$, if $P(n)
\subseteq C(n) \subseteq \rho Q(n)$ then any $\RR^m_+$-lift of $C(n)$ has 
$m = 2^{\Omega(n)}$. Therefore, one cannot
approximate $P(n)$ within a factor of $\rho$ by any polytope with a
small linear lift. This in turn says that the max
clique problem on a graph with $n$ vertices cannot be approximated
well by polytopes with small polyhedral lifts.  In fact they also
prove that even if $\rho=O(n^{\beta})$, for some $\beta \leq 1/2$, the size of $m$ 
grows exponentially.

The above result was proven in \cite{BraunFioriniPokuttaSteurer} by
showing that the {\em nonnegative rank} of the slack matrix $S_{P(n), \rho
  Q(n)}$, denoted as $\rank_+(S_{P(n), \rho
  Q(n)})$, has order $2^{\Omega(n)}$. The matrix $S_{P(n), \rho Q(n)}$
is a very particular perturbation of $S_{P(n),Q(n)}$. It is not hard to see
that the proof of Theorem 5 in \cite{BraunFioriniPokuttaSteurer} in
fact shows that all nonnegative matrices in a small neighborhood of
$S_{P(n),Q(n)}$ have high nonnegative rank.

 \begin{proposition}\label{prop:inapprox}
Let $P(n)$ and $Q(n)$ be as above, and $0<\eta<1/2$. For any
nonnegative matrix $S'(n)$ such that $\|S_{P(n),Q(n)}-S'(n)\|_{\infty}
\leq \eta$, $\rank_+(S'(n))=2^{\Omega(n)}$.
\end{proposition}

Since $\|S_{P(n),Q(n)}-S_{P(n),\rho Q(n)}\|_{\infty}=\rho-1$,
Proposition~\ref{prop:inapprox} implies a version of the first part of
the result in \cite{BraunFioriniPokuttaSteurer}, i.e., that the
nonnegative rank of $S_{P(n),\rho Q(n)}$ is exponential in $n$ (if
only for $1<\rho<3/2$). Further, Proposition~\ref{prop:inapprox} also
says that even if we allowed approximations of $P(n)$ in the sense of
Proposition~\ref{prop:errorbounds2} (i,e., approximations with a
$\RR^m_+ \times \SOC$-lift that do not truly nest between $P(n)$ and
$Q(n)$), $m$ would not be small. Therefore, the result in
\cite{BraunFioriniPokuttaSteurer} on the outer inapproximability of
$P(n)$ by small polyhedra is robust in a very strong sense.

\medskip
\textbf{Acknowledgements:} We thank Anirudha Majumdar for the
reference to~\cite{SturmZhang}, and Cynthia Vinzant for helpful
discussions about cone closures.

% Put this somewhere!
\nocite{SDOCAG}

\bibliographystyle{plain}
\bibliography{GPTnewbib}

\end{document}